\documentclass[11pt,english]{article}
\usepackage[utf8]{inputenc}
\usepackage[T1]{fontenc}
\usepackage{babel}
 \usepackage{amsmath}
 \usepackage{mathtools}
\usepackage{amsfonts}
\usepackage{amssymb}
\usepackage{graphicx}
\usepackage{stmaryrd}
\usepackage{amsthm}
\usepackage{bbm}
\usepackage{mathtools}
\usepackage[colorlinks=true,
            linkcolor=blue,
            urlcolor=blue,
            citecolor=blue]{hyperref}
\usepackage{authblk}
\usepackage{geometry}
\usepackage{algorithm}
\usepackage[noend]{algpseudocode}
\usepackage[vcentermath]{youngtab}
\usepackage{pgfplots}
\pgfplotsset{compat=newest}
\usetikzlibrary{calc}
\usepackage{mathrsfs}
\newtheorem{theorem}{Theorem}[section]%
\newtheorem{assumption}[theorem]{Assumption}%
\newtheorem{corollary}[theorem]{Corollary}%
\newtheorem{definition}[theorem]{Definition}%
\newtheorem{lemma}[theorem]{Lemma}%
\newtheorem{proposition}[theorem]{Proposition}%

\usetikzlibrary{automata, positioning}
\usepackage[normalem]{ ulem }

\newcommand{\bigslant}[2]{{\raisebox{.2em}{$#1$}\left/\raisebox{-.2em}{$#2$}\right.}}
\definecolor {processblue}{cmyk}{0.96,0,0,0}
 
\title{Small cycle structure for words in conjugation invariant random permutations}
\author{
Mohamed Slim Kammoun\footnote{Université de  Toulouse - Institut de Mathématiques de Toulouse, France and Université de Lyon, ENSL, CNRS, France.
Email: \href{mailto:slim.kammoun@ens-lyon.fr}{\nolinkurl{slim.kammoun@ens-lyon.fr}}
}
\quad
Myl\`ene Ma\"ida\footnote{
 Univ. Lille, CNRS, UMR 8524 - Laboratoire Paul Painlevé, F-59000 Lille, France and IRL CRM-CNRS, International Reserach Lab, Montréal, Canada.
 Email: \href{mailto:mylene.maida@univ-lille.fr}{\nolinkurl{mylene.maida@univ-lille.fr}}. 
}
}\date{\today}
\newcommand\mps[1]{}
\begin{document}

\maketitle

\begin{abstract}
We study the cycle structure of words in several random permutations. We assume that the permutations are independent and that their distribution is conjugation invariant, with a good control on their short cycles. If,
after successive cyclic simplifications,  the word $w$ still contains at least two different letters, then  we get a universal limiting joint law for short cycles for the word in these permutations. These results can be seen as an extension of our previous work \cite{KaMa20} from the product of permutations to any non-trivial word in the permutations and also as an extension of the results of \cite{Nica} from uniform permutations to general conjugation invariant random permutations. In particular, we get optimal assumptions in the case of the commutator of two such random permutations.
\end{abstract}

\section{Introduction and statement of the results}

%We study the cycle structure of words in several random permutations. We assume that the permutations are independent and that their distribution is conjugation invariant, with a good control on their short cycles. If,
%after successive cyclic simplifications,  the word $w$ still contains at least two different letters, then  we get a universal limiting joint law for short cycles for the word in these permutations. These results can be seen as an extension of our previous work \cite{KaMa20} from the product of permutations to any non-trivial word in the permutations and also as an extension of the results of \cite{Nica} from uniform permutations to general conjugation invariant random permutations. In particular, we got optimal assumptions in the case of the commutator of two such random permutations.

The object of this paper is to study the behavior of the cycles up to a fixed size for a word in several independent conjugation invariant random permutations.

We start by recalling that the asymptotics of the  short cycle structure of one  permutation chosen uniformly in the symmetric group is well known and Poisson.
If we denote by $\mathfrak{S}_{n}$ the symmetric group of size $n \in \mathbb N^*,$ and  by $\#_\ell \sigma$  the number of cycles of length $\ell$ of the permutation $\sigma,$ we have :
\begin{theorem}\cite[Theorem 3.1]{ABT00}
For $\rho_n$ following the uniform distribution on  $\mathfrak{S}_{n},$  for any $d' \ge 1,$ 
 $$ (\#_1 \rho_{n},\#_2 \rho_{n},\ldots,\#_{d'} \rho_{n}) \xrightarrow[n \rightarrow \infty]{(d)} (\xi_1, \ldots, \xi_{d'}),$$ 
where $\xrightarrow[n \rightarrow \infty]{(d)}$ is the convergence in distribution and $(\xi_1, \ldots, \xi_{d'})$ are independent variables, where the distribution of $\xi_\ell$ is Poisson with parameter $\frac{1}{\ell}.$

Moreover, the joint moments converge and in particular,
$$   \mathbb E\left(\#_1 \rho_{n}\right) \xrightarrow[n \rightarrow \infty]{} 1.$$ 
\end{theorem}

It is also known that the asymptotic short cycle structure for words in independent uniformly distributed permutations is Poisson.
More precisely, let {$k>1$} be fixed  and $F_k := \langle x_1, \ldots, x_k\rangle$ be the free group on $k$ generators $\{x_1, \ldots, x_k\}.$ \emph{In the sequel, the index $k$ will always refer to the number of generators.} An element $w$ of $F_k$ can be written
\[  w= x_{k_r}^{\varepsilon_r}\ldots x_{k_1}^{\varepsilon_1},\]
with $\varepsilon_j = \pm 1$ and if $k_{j+1}=k_j$ then $\varepsilon_{j+1}=\varepsilon_{j}$. Such a form is unique, we call it in the sequel the \emph{canonical form} of the word $w.$

%For any word $w$ and $d \in \mathbb N^*,$ $w^d$ is the concatenation of $d$ copies of $w.$ Note that any word $w\neq 1$ can be written $w=\Omega^d$ where $\Omega$ is not a power and $d\geq 1$. We say that a word $w$ is not a power if it cannot be written as $\Omega^d,$ for some $d\ge 2.$ 

For every group $G,$ $w = x_{k_r}^{\varepsilon_r}\ldots x_{k_1}^{\varepsilon_1}$ induces a word map from the Cartesian product $G^k$ into $G$ by substitution :
\[ \begin{array}{rrcl}
   w: &    G^k & \rightarrow & G\\
      & (g_1, \ldots, g_k) & \mapsto & g_{k_r}^{\varepsilon_r}\ldots g_{k_1}^{\varepsilon_1}.
      \end{array}
\]

We consider the case when $G= \mathfrak{S}_{n},$ the symmetric group of size $n \in \mathbb N^*,$ equipped with its composition law. 
When $g_1, \ldots, g_k$ is a $k$-tuple of  random permutations, their image by the word map $w$ is again a random permutation and one can study the distribution of its cycle structure. The more natural case to consider is independent uniformly distributed random permutations. It has been investigated by Nica, and more recently by Puder and coauthors. 
They have shown that the limiting distribution depends only on the algebraic structure of the word $w.$ For any word $\Omega,$ and $d \in \mathbb N^*,$ $\Omega^d$ denotes the composition of $d$ copies of the word $\Omega.$ We say that the word $w$ is a \emph{power} if there exists a word $\Omega \neq 1$ and $d\ge 2$ such that $w= \Omega^d.$   Note that any word $w\neq 1$ can be written $w=\Omega^d$ where $\Omega$ is not a power and $d\geq 1$. We have
the following :

\begin{proposition}
\label{prop:pp} \cite[Theorem 1.3]{MR4138706} \cite[Theorem 1.1]{Nica}
Let $w=\Omega^d\neq 1$ be a word in $F_k$ such that  $d\geq 1$ and $\Omega$ is not a power.  Let $\rho_{1,n}, \ldots, \rho_{k,n}$ be independent and \textbf{uniformly distributed} on $ \mathfrak{S}_{n}.$ Then the random vector       $(\#_1 w(\rho_{1,n}, \ldots, \rho_{k,n}), \rho_{k,n}),\ldots,\#_{d'} w(\rho_{1,n}, \ldots, \rho_{k,n}))$
    converges in distribution  to a limit that depends only on the power $d$ and the maximal length $d'$ of the cycles under consideration. Moreover, the joint moments converge also to the corresponding moments of the limiting law.  In particular,
\begin{align} \label{equation:Expectation_of_fixed_points}
    \mathbb E\left(\#_1 w(\rho_{1,n}, \ldots, \rho_{k,n})\right) \xrightarrow[n \rightarrow \infty]{} \psi(d)
\end{align}
where $\psi(d):=\sum_{\ell|d} 1$ is the number of divisors of the integer $d$.
%the expectation  of the number of fixed points of the permutation 
%$w(\sigma_{1,n}, \ldots, \sigma_{k,n})$ converges to 1 as $n$ grows to infinity.
\end{proposition}
Remarkably, the limiting law does not
 depend on $\Omega.$ In particular, the limiting law is  the same  if we replace the word  $\Omega$ by the word  $x_1$ in just one letter.

Our goal in this paper will be to explore universality for these beautiful results. Before stating ours, let us give a few motivations 
to explore words in random permutations. The initial motivations of \cite{Nica} was related to free probability theory and the notion of asymptotic freeness. Without going too much into the details, one can say that for a $k$-tuple of random matrices $X_{1, N}, \ldots, X_{k,N}$ with a prescribed joint law, their asymptotic joint distribution in the framework of free probability theory is determined by the limit of $\mathbb E(\frac{1}{N}\mathrm{Tr}\,w(X_{1, N}, \ldots, X_{k,N})),$ for any word $w.$ If the random matrices are independent and their distribution is unitarily invariant, the behavior of the limit is well known and related to the notion of freeness - one says that such random matrices are asymptotically free.  It was natural in this context to ask the same question for permutation matrices. This is the point of view of Nica. As the trace of a permutation matrix is related to the number of fixed points, one can see that the question of asymptotic freeness in this case is closely related to the study of the short cycle structure.
In the same  direction, one can also mention the theory of traffic distribution, developed by Male and coauthors (see e.g. \cite{Ma11, AuMa20}), exploring the case when unitary invariance is replaced by permutation invariance and the notion of second order freeness introduced by Mingo and Speicher (see e.g. \cite{MS,Mal2}. Talking about links with random matrix theory, \cite{Dubach} studies the total number of cycles for  the commutator of two independent uniform permutations by using a connection with the moments in non-Hermitian random matrices.

\noindent
In  a series of papers \cite{PuPar15, MR4138706, MaPu19},
Puder and coauthors use the properties of the distribution of some statistics of $w(\rho_{1,n}, \ldots, \rho_{k,n}),$ as the number of fixed points, short cycles etc. to characterize the algebraic properties 
of the word $w,$ through what they call the primitivity rank of the word $w.$ Understanding $w(\rho_{1,n}, \ldots, \rho_{k,n})$ can then give access to properties of the word map for other groups $G.$

In the present work, the goal is to extend Proposition \ref{prop:pp}
to a {$k$}-tuple $(\sigma_{1,n}, \ldots, \sigma_{k,n})$ of \textbf{independent} and \textbf{conjugation invariant} random permutations. 
\begin{definition}
For any $n \ge 1,$ let $(\sigma_{1,n}, \ldots, \sigma_{k,n})$ be random permutations in $(\mathfrak S_n)^k.$
We say that the (distribution of the) $k$-tuple $(\sigma_{1,n}, \ldots, \sigma_{k,n})$  is conjugation invariant if  for any fixed permutations $\sigma_1, \ldots, \sigma_k \in \mathfrak{S}_{n},$ $(\sigma_1^{-1}\circ \sigma_{1,n} \circ \sigma_1, \ldots, \sigma_k^{-1}\circ \sigma_{k,n} \circ \sigma_k)$ has the same distribution as $(\sigma_{1,n}, \ldots, \sigma_{k,n}).$
\end{definition}

One can give several natural motivations to extend the results known for uniform permutations in this direction. First, if we go back to the question of asymptotic freeness mentioned above, seeking for universality for these questions has always been a strong motivation. For example, the work \cite{BeGe10} is clearly presented as an extension of \cite{Nica} with motivations arising from free probability theory and studies some particular conjugation invariant distributions. It is also worth mentioning a connection with random maps : in \cite{BuCuPe19} for example, the study of random maps is strongly related to the study of the cycle structure of the product of two permutations, one with prescribed cycle structure (corresponding to the faces of the map), the other being an involution without fixed points (corresponding to the edges), both naturally inheriting conjugation invariant distributions.

Before stating our theorem, let us detail our assumptions on random permutations having few short cycles.

\begin{assumption}
Let $S_1, S_2$ be two finite subsets of $\mathbb N^*.$ Let $(\sigma_n)_{n \in \mathbb N^*}$ be a sequence of random permutations (with for all $n \ge 1, \sigma_n \in \mathfrak S_n$). 
We say that the sequence $(\sigma_n)_{n \in \mathbb N^*}$ has few short cycles with respect to $(S_1, S_2)$ 
\begin{itemize}
    \item in the weak sense if 
    \begin{equation}
 \forall i \in   S_1, \frac{\#_{i} \sigma_{n}}{\sqrt{n}} \xrightarrow[n\to\infty]{\mathbb P}0,
\end{equation}
and
\begin{equation}
 \forall i \in S_2, \frac{\#_{i} \sigma_{n}}{{n}} \xrightarrow[n\to\infty]{\mathbb P}0,
\end{equation}
where $\xrightarrow[n\to\infty]{\mathbb P}$ stands for the convergence in probability,
  \item   and in the strong sense if
\begin{equation}\label{hyp:prim}
\forall p \in \mathbb N^*, c_1, \ldots, c_p \in S_1, \lim_{n\to\infty} \mathbb E\left( \prod_{i=1}^p \left(\frac{\#_{c_i} \sigma_{n}}{\sqrt{n}}\right)\right) =0
\end{equation}
and
\begin{equation}\label{hyp:prim_2}
\forall i \in S_2, \lim_{n\to\infty} \mathbb E \left(\frac{\#_{i} \sigma_{j,n}}{{n}}\right) =0.
\end{equation}    
\end{itemize}

\end{assumption}

Otherwise stated, we ask for a good probabilistic control on some cycles of length less than the maximum of $S_1 \cup S_2.$
This control may be in probability for the weak version or on in mixed moments for the strong version. 

In the assumptions of our results, we will need to choose the sets $S_1$ and $S_2$ according to the word $w$ we are interested in. We detail now this choice.

%We now go to the statements of our results. The goal is to extend Proposition \ref{prop:pp}
%to a {$k$}-tuple $(\sigma_{1,n}, \ldots, \sigma_{k,n})$ of \textbf{independent} and \textbf{conjugation invariant} random permutations. We say that the $k$-tuple $(\sigma_{1,n}, \ldots, \sigma_{k,n})$ is conjugation invariant if  for any fixed permutations $\sigma_1, \ldots, \sigma_k \in \mathfrak{S}_{n},$ $(\sigma_1^{-1}\circ \sigma_{1,n} \circ \sigma_1, \ldots, \sigma_k^{-1}\circ \sigma_{k,n} \circ \sigma_k)$ has the same distribution as $(\sigma_{1,n}, \ldots, \sigma_{k,n}).$ 
%One can give several natural motivations to extend the results known for uniform permutations in this direction. First, if we go back to the question of asymptotic freeness mentioned above, seeking for universality for these question has always been a strong motivation. For example, the work \cite{BeGe10} is clearly presented as an extension of \cite{Nica} with motivations arising from free probability theory and studies some conjugation invariant distributions. It is also worth mentioning the connection with random maps : in \cite{BuCuPe19} for example, the study of random maps is strongly related to the study of the cycle structure of the product of two permutations, one with prescribed cycle structure (corresponding to the faces of the map) and one involution without fixed points (corresponding to the edges), both naturally having conjugation invariant distributions.  

%To state our results, we need to introduce some notations.
For a given word $w$ in $F_k,$ we will alternatively use its canonical form - that we have already introduced above~: $  w= x_{k_r}^{\varepsilon_r}\ldots x_{k_1}^{\varepsilon_1},$
with $\varepsilon_j = \pm 1$ and if $k_{j+1}=k_j$ then $\varepsilon_{j+1}=\varepsilon_{j}$ - and its \emph{reduced form}. Any word $w \in F_k$ can be written $w:=x_{\alpha_1}^{\beta_{1}}\dots x_{\alpha_\ell}^{\beta_{\ell}}$  with $\beta_i \in \mathbb{Z}^*$, $\ell \ge 1$,  $\alpha_i \in \llbracket k\rrbracket$,  $\alpha_{i+1}\neq \alpha_i$,   where $\llbracket k \rrbracket = \{1, \ldots, k\}.$ 
  We call this rewriting the \emph{reduced form} of the word.

  We  define the following sets of integers : for any $w\in F_k$ with reduced form $x_{\alpha_1}^{\beta_{1}}\dots x_{\alpha_\ell}^{\beta_{\ell}}$ and $1\leq j \leq k$, let   $${\mathcal N'_j}(w) :=\{ |\beta_i|,  \alpha_i=j \},\widehat{ \mathcal N}_{j}(w):=\{ |d|: \exists \beta\in \mathcal N'_j(w), d|\beta  \} \textrm{ and } \mathcal N_j(w):=\llbracket  \max \widehat{ \mathcal N}_j(w) \rrbracket  \cup \{ { \mathcal N'_j}(w)+{ \mathcal N'_j}(w) \}.$$

  Note that $ \mathcal N_j(w) \subset \llbracket 2\max (\widehat{\mathcal N}_{j}(w)) \rrbracket = \llbracket 2\max (\mathcal N^\prime_{j}(w)) \rrbracket.$\\
  
  For example, for $w=x_1x^3_2x_1^{-1}x_2^6,$ $\mathcal N'_1(w) = \widehat {\mathcal N}_1(w) =\{ 1\}, \mathcal N_1(w) =\{1,2\},$ 
  $\mathcal N'_2(w) = \{3,6\}, \widehat {\mathcal N}_2(w) =\{1,2,3,6\}, \mathcal N_2(w) =\{1,2,3,4,5,6,9,12\}.$\\

%To state our result, we need some more notations. For any permutation $\sigma \in \mathfrak{S}_{n},$
%$c_m(\sigma)$ denotes the length of the cycle containing $m$, $\#_\ell \sigma$ is the number of cycles of length $\ell$ and $\# \sigma$ the total number of cycles
%of $\sigma.$ 
%We also need the notion of \emph{reduced word}.  Any word $w \in F_k$ can be written $w:=x_{\alpha_1}^{\beta_{1}}\dots x_{\alpha_\ell}^{\beta_{\ell}}$  with $\beta_i \in \mathbb{Z}^*$, $\ell>1$,  $\alpha_i \in \llbracket k\rrbracket$,  $\alpha_{i+1}\neq \alpha_i$. It is well known that the number of cycles of a given length of a permutation can be expressed in terms of traces of powers of the matrix associated to the permutation. Consequently, by cyclic invariance of the trace, one can suppose in addition  that  $\alpha_\ell \neq \alpha_1$.  We call this rewriting the \emph{reduced form} of the word. Then, for any $j \in \llbracket  k \rrbracket,$ we define $\gamma^i(w) := (|\beta_j|)_{\alpha_j=i}.$ 

Our main result is the following extension of Proposition \ref{prop:pp}.
{
\begin{theorem}\label{th:conv_dis}
Let  $w$ be a word in  $F_k$ with canonical form $w= x_{k_r}^{\varepsilon_r}\ldots x_{k_1}^{\varepsilon_1},$  where $k_1 \neq k_r$ (there is at least two different letters). We write $w= \Omega^d,$ with   $\Omega$ which is not a power and $d \ge 1.$  

Let $\sigma_{1,n}, \ldots, \sigma_{k,n}$ be independent random permutations, with conjugation invariant distributions 
such that $\forall j \in \llbracket k\rrbracket,$ the sequence $(\sigma_{j,n})_{n \ge 1}$ has few short cycles with respect to $(\widehat { \mathcal  N}_j(w), \mathcal N_j(w))$ in the strong sense.

%satisfying the following assumptions :
%\begin{equation}\label{hyp:prim}
%\forall j \in \llbracket k\rrbracket,  \forall p \in \mathbb N^*, c_1, \ldots, c_p \in \widehat { \mathcal  N}_j(w), \lim_{n\to\infty} \mathbb E\left( \prod_{i=1}^p \left(\frac{\#_{c_i} \sigma_{j,n}}{\sqrt{n}}\right)\right) =0
%\end{equation}
%and
%\begin{equation}\label{hyp:prim_2}
%\forall j \in \llbracket k\rrbracket, \forall i \in \mathcal N_j(w), \lim_{n\to\infty} \mathbb E \left(\frac{\#_{i} %\sigma_{j,n}}{{n}}\right) =0.
%\end{equation}
Then, for any $d^\prime \in \mathbb N^*,$
$(\#_{1} w(\sigma_{1,n}, \ldots, \sigma_{k,n}),\dots,\#_{d^\prime} w(\sigma_{1,n}, \ldots, \sigma_{k,n}))$ converges in distribution to a universal limit that depends only on the power $d$ and the maximal length $d^\prime$ of the cycles under consideration. Moreover,  the joint moments  converge also to those of the limit.
\end{theorem}}

The convergence obtained in Theorem \ref{th:conv_dis} is a convergence 
in distribution and for all the moments. By standard probabilistic arguments, that will be briefly presented at the end of Section \ref{sec:proof2}, the assumptions can be weakened to get the following corollary. 

\begin{corollary}\label{corollary_dis} 
Let  $w$ be a word in  $F_k$ with canonical form $w= x_{k_r}^{\varepsilon_r}\ldots x_{k_1}^{\varepsilon_1},$  where $k_1 \neq k_r$ (there is at least two different letters). We write $w= \Omega^d,$ with   $\Omega$ which is not a power and $d \ge 1.$  

Let $\sigma_{1,n}, \ldots, \sigma_{k,n}$ be independent random permutations, with conjugation invariant distributions 
such that $\forall j \in \llbracket k\rrbracket,$ the sequence $(\sigma_{j,n})_{n \ge 1}$ has few short cycles with respect to $(\widehat { \mathcal  N}_j(w), \mathcal N_j(w))$ in the weak sense.

Then, for any $d^\prime \in \mathbb N^*,$ the random vector
$(\#_{1} w(\sigma_{1,n}, \ldots, \sigma_{k,n}),\dots,\#_{d^\prime} w(\sigma_{1,n}, \ldots, \sigma_{k,n}))$ converges in distribution to a universal limit that depends only on the power $d$ and the maximal length $d^\prime$ of the cycles under consideration. \end{corollary}

Let us  first comment on the condition, $k_1\neq k_r$.
For example, if $w=x_1x_2x_1^{-1}$, $w(\sigma_1,\sigma_2)$ has the same cycle structure as $\sigma_2$ so obviously there is no universality.

One can easily check the following cyclic invariance : for any $\boldsymbol \sigma=(\sigma_1, \ldots, \sigma_k) \in \mathfrak S^k_n$ and $w \in F_k,$ for any $\ell \ge 1, i \in \llbracket  k \rrbracket,$
\[ \#_\ell (\sigma_i w(\boldsymbol\sigma)) = \#_\ell ( w(\boldsymbol\sigma)\sigma_i).\] Using this remark, one can check  that when $k_1 = k_r,$ there are 3 cases : 
\begin{itemize}
    \item either $w=1$ : in this case there is nothing to prove since $w(\boldsymbol \sigma)=I_n$, that fixes all integers,
    \item or $w= w_1x_{j}^\alpha w_1^{-1}$ : in this case, the cycle structure of $w(\boldsymbol \sigma)$ is the same as that of $\sigma_j^\alpha$ and there is no universality,
    \item or $w=w_1w_2$ with $w_1$ and $w_2$ two words  such that the canonical form of $w':=w_2w_1=x_{{k}^\prime_{r^\prime}}^{\varepsilon^\prime_{r^\prime}}\ldots x_{k^\prime_1}^{\varepsilon^\prime_1}$ satisfies $k^\prime_1\neq k^\prime_{r^\prime}.$ In this case, the cycle structure of $w^\prime$ is the same  as that of $w$.
\end{itemize}
In other words, if, after successive cyclic simplifications,  the word $w$ still contains at least two different letters, then, by applying this word to independent  permutations with conjugation invariant distributions and few short cycles,  we get a random permutation with an universal limiting joint distribution for short cycles. 

Similar results have already been proved  in the case
when the permutations are uniformly chosen among permutations with restrictions on cycle lengths  in \cite[Theorem 3.7]{BeGe10} (the latter being an extension of \cite{Neagu}, in which only the fixed points of the words in permutations were studied). They can also be seen as an extension of our previous work \cite{KaMa20} for the product of permutations to any non-trivial word in the permutations. In the case of the product, we could get optimal assumptions, requiring the permutations to have only few fixed points and cycles of size 2.
These assumptions were reminiscent of the connectivity assumptions for random maps appearing in \cite{BuCuPe19}. In the case when $|\beta_i|=1,$ for all $i \le \ell,$ we recover the exact same assumptions. 
For general words, we nevertheless require a more stringent control on short cycles, that may not be optimal.  Optimality is discussed in Section \ref{sec:opt}, it does not hold in full generality but we can show that the assumptions we got for the product are also optimal in the case of the commutator of two permutations.

Note that in this work we focus only on short cycles but in the literature, there is also interest in the limit shape of the Young diagram associated to the cycle structure of a random permutation. We refer to \cite{MR3457548} for further details and exact definitions. This kind of convergence requires also the study of long cycles, which is an interesting question beyond the scope of this study. %The short cycles are encoded by the edge of the Young diagram. 
In the framework  of conjugation invariant permutations, short cycles are still very interesting  since they gives access to local events i.e.
if $\sigma_n$ and $\rho_n$ are conjugation invariant random permutations such that 
$(\#_1\,\sigma_n, \ldots, \#_\ell\,\sigma_n) \overset{d}= (\#_1\,\rho_n, \ldots, \#_\ell\,\rho_n)$
for some  $\ell<\frac n2$, then for any $1\leq i_1,i_2,\ldots,i_\ell \leq n$ and  for any $1\leq j_1,j_2,\ldots,j_\ell\leq n$, the equality
$$ \mathbb{P}(\sigma_n(i_1)=j_1,\ldots \sigma_n(i_\ell)=j_\ell) =\mathbb{P}(\rho_n(i_1)=j_1,\ldots \rho_n(i_\ell)=j_\ell)  $$
holds true.  This remark is a direct consequence, for example, of \cite[Equation (5.68)]{hamaker2022characters} (see also \cite{diaconis1989generalization}).  Theorem \ref{th:conv_dis} can then be seen as a local convergence result. \\ \ 
%To the best of our knowledge, the result is not written in this way anywhere, but can also be derived from earlier works 
%since $\#_i = \lambda'_{i+1}-\lambda'_i{i}$ where $\lambda'_i$ is the length of the $n^th$ column of $\lambda$ the young diagram encoding the cycle structure of the permutation. \\
  
\noindent
\textbf{Acknowledgements:} Both authors would like to thank Maxime Février for very useful discussions. S.K. would also like to thank Zachary Hamaker and Camille Male for bibliographical help.
   M.M. is partially supported by the Labex CEMPI (ANR-11-LABX-0007-01) and S.K. is partially supported by a Leverhulme Trust Research Project Grant (RPG-2020-103), Labex CIMI (ANR-11-LABX-0040) and an ERC Project LDRAM (ERC-2019-ADG Project 884584).

\section{Technical lemmas}

The technique of proof is inspired by \cite{KaMa20}. It is based on a graphical representation of the image of several points by the permutations composing the word $w.$ 
The strategy of the proof will be to identify the class of graphs giving the main contribution in the uniform case and then to show that, under our assumptions, the non vanishing contributions are the same in the conjugation invariant case.\\

%\section{Proof of Theorem \ref{th:general}}

Let us introduce some definitions. Some of them have already been used in \cite{KaMa20} but we will need  colored versions of the graphs used therein.
\begin{itemize}

 \item We denote by $\mathbb G_\ell^n$ the set of oriented  graphs with vertices $\llbracket n\rrbracket$  having exactly $\ell$ (oriented) edges and  $\mathbb G^n$ the set of oriented  graphs with vertices $\llbracket n\rrbracket.$ We allow here loops but not multiple edges.

 We denote also by $\mathbb H_\ell^n$ the set of oriented  graphs with vertices $\llbracket n\rrbracket$  each colored by red or white and   having exactly $\ell$ (oriented) edges and  $\mathbb H^n$ the set of  oriented  graphs with vertices $\llbracket n\rrbracket,$  each colored by red or white. 
 Given  a (possibly colored) graph $g$ we denote by $E_g$ the set of its edges. %and by $A_g:= (\mathbf 1_{(i,j)\in E_g})_{1 \le i,j \le n}$ its adjacency matrix.
 \item Given $g\in \mathbb{G}^n$ and $h\in \mathbb{H}^n$, we say that $h$ is of type $g$ and we denote it by $\dot{h}=g$ if $E_g=E_h$ i.e. $g$ is the non-colored version of $h$.
 
 \item For any $\sigma \in \mathfrak S_n,$ we denote by $g_{\sigma}$ the graph with vertices  $\llbracket n\rrbracket$ and oriented edges 
 
 \[   \bigcup_{\ell=1}^n \left\{\left( \ell, \sigma(\ell)\right) \right\}.\]
 
 \item Given $g \in \mathbb G^n,$ we denote by 
 \[\mathfrak S_{n,g}:=\left\{\sigma \in \mathfrak S_n, \forall (i,j)\in E_g, \sigma(i) = j\right\}.\]
 In other words, $\mathfrak S_{n,g}$ is the set of permutations $\sigma$ such that $g$ is a subgraph of $g_\sigma.$
 
\item A vertex $i$ of $g$ is called \textit{isolated} if $E_g$ does not contain any edge of the form $(i,j)$ or $(j,i)$ nor a loop $(i,i)$. Let $g\in\mathbb{G}^n $, we denote by  $\tilde{g}$  the graph obtained  from  $g$  after removing isolated vertices.

\item A connected component of $g$ is called \textit{trivial} if it is reduced to one isolated vertex.

\item Two graphs (resp. colored graphs) are isomorphic if there exists a permutation of their vertices that preserves edges (resp. edges and colors). 
Let $\mathcal{R}$ be the equivalence relation  on colored graph such that $h_1\mathcal{R} h_2$ if  $\tilde h_1$ and $\tilde h_2$ are two colored graphs that are isomorphic and $\dot{\mathcal{R}}$ be the equivalence relation  on non-colored graph such that $g_1\dot{\mathcal{R}} g_2$ if  $\tilde g_1$ and $\tilde g_2$ two graphs that are isomorphic. 
 We denote by  $\widehat{\mathbb{H}}_\ell:=\bigslant{\cup_{n\geq1} \mathbb{H}_\ell^n}{\mathcal{R}}
$  and  $\widehat{\mathbb{G}}_\ell:=\bigslant{\cup_{n\geq1} \mathbb{G}_\ell^n}{\dot{\mathcal{R}}}
$   the respective set of  equivalence classes of $\cup_{n\geq1} \mathbb{H}_\ell^n$ and $\cup_{n\geq1} \mathbb{G}_\ell^n$  and we set $  \widehat{\mathbb{H}} := \bigcup_{\ell \ge 1}  \widehat{\mathbb{H}}_\ell $ and
$  \widehat{\mathbb{G}} := \bigcup_{\ell \ge 1}  \widehat{\mathbb{G}}_\ell .$

\end{itemize}

We will now introduce the graphs useful for our purpose. We start from an integer $m$ and look at its orbit along the word $w(\sigma_1, \ldots, \sigma_k).$  If the reduced form of the word is $w= x_{\alpha_\ell}^{\beta_\ell} \ldots x_{\alpha_1}^{\beta_1},$ we perform from $m$ a walk of length $|\beta_1|$ on the graph $g_{\sigma_{\alpha_1}},$
then a walk of length $|\beta_2|$ on the graph $g_{\sigma_{\alpha_2}},$ etc. This will provide the vertices and edges of the $k$ graphs corresponding to each permutation and each vertex is colored red if it is the entering or exiting point of the walks (see the definitions and the proof of Lemma \ref{graph_geometry} below for more details). In comparison to \cite{KaMa20}, considering graphs colored this way will allow more accurate control to show that the graphs containing loops do not contribute to the limiting distribution.

\begin{itemize}

\item Let $n \in \mathbb N^*$ and $\boldsymbol\sigma =(\sigma_1, \ldots, \sigma_k) \in (\mathfrak S_n)^k.$ Let $m \in \llbracket n\rrbracket$ be fixed. 
 For a word with canonical form $w=  x_{k_r}^{\varepsilon_r}\ldots x_{k_1}^{\varepsilon_1},$ we define
 \[ i_0^{m, w}(\boldsymbol\sigma):=m,\,  i_1^{m, w}(\boldsymbol\sigma):= \sigma_{k_1}^{\varepsilon_1}(m),\, \ldots,\, i_r^{m, w}(\boldsymbol\sigma):= \sigma_{k_r}^{\varepsilon_r}\ldots \sigma_{k_1}^{\varepsilon_1}(m)=w(\boldsymbol{\sigma})(m), \]
for any $a \ge 1 $ and $0 \le b \le r-1,$
\[  i_{ar+b}^{m, w}(\boldsymbol\sigma):= w^a(\boldsymbol{\sigma})(i_{b}^{m, w}(\boldsymbol\sigma)).\]

We also define the colors of the vertices as follows : the color of the vertex $j$ with respect to the permutation $\sigma_i$ is given by
\[\mathfrak c^{m,w}_{i,j}(\boldsymbol\sigma)=  \text{red},  \text{if } \begin{cases}
  j=m \text{ and } i={k_1}    \\ \text{or } j= w(\boldsymbol{\sigma})(m)  \text { and } i={k_r}
 \\ \text{or } \exists \ell \text{ such that } j=i_\ell^{m, w}(\boldsymbol\sigma),  i\in\{k_{\ell},k_{\ell+1}\} \text{ and } k_\ell \ne k_{\ell+1}  
 \end{cases},\]  and white otherwise.
 
 For any $i \in \llbracket k\rrbracket,$ $\mathcal G_i^{m, w}(\boldsymbol\sigma) \in \mathbb{G}^n $ is the graph with vertices  $\llbracket n\rrbracket$ and edges 
 
 \[  E_{ \mathcal G_i^{m, w}(\boldsymbol\sigma)} := \left(\bigcup_{\{\ell : k_\ell = i, \varepsilon_\ell=1 \}} \left\{\left( i_\ell^{m, w}(\boldsymbol\sigma), i_{\ell+1}^{m, w}(\boldsymbol\sigma)\right) \right\}\right) \bigcup \left( \bigcup_{\{\ell : k_\ell = i, \varepsilon_\ell=-1 \}} \left\{\left( i_{\ell+1}^{m, w}(\boldsymbol\sigma), i_{\ell}^{m, w}(\boldsymbol\sigma)\right) \right\} \right), \]

 $\mathcal H_i^{m, w}(\boldsymbol\sigma) \in \mathbb{H}^n$  is the graph with vertices  $\llbracket n\rrbracket,$ the color of vertex $j$ being $c^{m,w}_{i,j}(\boldsymbol\sigma),$ and edges  $ E_{ \mathcal G_i^{m, w}(\boldsymbol\sigma)}.$ Finally, let $ {\widehat{\mathcal G}}_i^{m, w}(\boldsymbol\sigma)$ and $ {\widehat{\mathcal H}}_i^{m, w}(\boldsymbol\sigma)$ be respectively the equivalence class of   $\mathcal G_i^{m, w}(\boldsymbol\sigma)$ and $\mathcal H_i^{m, w}(\boldsymbol\sigma).$
 
 Given $s\subset\llbracket n \rrbracket$, let  $ {\mathcal G}_i^{s, w}(\boldsymbol\sigma)$ be the graph such that  
 $E_{ {\mathcal G}_i^{s, w}(\boldsymbol\sigma)}=\bigcup_{m\in s} E_{ {\mathcal G}_i^{m, w}(\boldsymbol\sigma)}$ and let $ {\mathcal H}_i^{s, w}(\boldsymbol\sigma)$ be the graph such that  
 $E_{ {\mathcal H}_i^{s, w}(\boldsymbol\sigma)}=\bigcup_{m\in s} E_{ {\mathcal H}_i^{m, w}(\boldsymbol\sigma)}$ and the color of a vertex $j$ is red  in $ {\mathcal H}_i^{s, w}(\boldsymbol\sigma)$ if and only if the color of $j$ is red in $ {\mathcal H}_i^{m, w}(\boldsymbol\sigma)$ for some $m\in s$.
 
 %Remark :  $\mathcal G_i^{s, w}(\boldsymbol\sigma)$ is a subgraph of $g_{\sigma_i}$.  

For example, if $w=(x_1^2x_2^3)^2$, $\sigma_1=(4,2,3,8)(1,7)(5,6)$ and $\sigma_2= (1,2,3,4,5)(9,10),$ we have 
 
 \begin{center}
    
 $ \mathcal G_1^{1, w}(\boldsymbol\sigma) = $
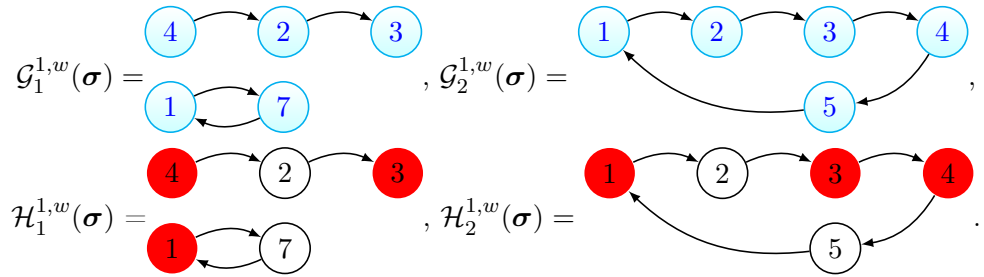
\begin {tikzpicture} [baseline={([yshift={-30pt}]current bounding box.north)},-latex,auto,node distance =1 cm and 1.5cm,on grid,
semithick,
state/.style ={ circle,top color =white, bottom color = processblue!20,
draw,processblue, text=blue, minimum width =0.1 cm}]
\node[state] (C) {$1$};
\node[state] (A) [above right=of C] {$2$};
\node[state] (A2) [ right=of A] {$3$};
\node[state] (B) [above =of C] {$4$};
\node[state] (D) [right =of C] {$7$};
\path (B) edge [bend left =25]  (A);
\path (A) edge [bend left =25]  (A2);
\path (C) edge [bend left =25]  (D);
\path (D) edge [bend left =25]  (C);

%\path (E) edge [bend left =25]  (F);
%\path (F) edge [bend left =25]  (E);

%\node[] at (-1.5,0) {$\mathcal{G}^1_1(\sigma,\rho)=$};
\end{tikzpicture},
$\mathcal G_2^{1, w}(\boldsymbol\sigma)=$  \begin {tikzpicture} [baseline={([yshift={-30pt}]current bounding box.north)},-latex,auto,node distance =1 cm and 1.5cm,on grid,
semithick,
state/.style ={ circle,top color =white, bottom color = processblue!20,
draw,processblue, text=blue, minimum width =0.1 cm}]
\node[state] (A)  {$1$};
\node[state] (B) [ right=of A] {$2$};
\node[state] (C) [ right=of B] {$3$};
\node[state] (D) [ right=of C] {$4$};
\node[state] (E) [below=of C] {$5$};
\path (A) edge [bend left =25]  (B);
\path (B) edge [bend left =25]  (C);
\path (C) edge [bend left =25]  (D);
\path (D) edge [bend left =25]  (E);
\path (E) edge [bend left =25]  (A);
%\path (E) edge [bend left =25]  (F);
%\path (F) edge [bend left =25]  (E);

%\node[] at (-1.5,0) {$\mathcal{G}^1_1(\sigma,\rho)=$};
\end{tikzpicture},

 $\mathcal H_1^{1, w}(\boldsymbol\sigma)$ =\begin {tikzpicture} [baseline={([yshift={-30pt}]current bounding box.north)},-latex,auto,node distance =1 cm and 1.5cm,on grid,
semithick,
state/.style ={ circle,top color =white, bottom color = white,
draw,black, text=black, minimum width =0.1 cm}]
\node[fill=red,circle] (C) {$1$};
\node[state] (A) [above right=of C] {$2$};
\node[fill=red,circle] (A2) [ right=of A] {$3$};
\node[fill=red,circle] (B) [above =of C] {$4$};
\node[state] (D) [right =of C] {$7$};
\path (B) edge [bend left =25]  (A);
\path (A) edge [bend left =25]  (A2);
\path (C) edge [bend left =25]  (D);
\path (D) edge [bend left =25]  (C);

%\path (E) edge [bend left =25]  (F);
%\path (F) edge [bend left =25]  (E);

%\node[] at (-1.5,0) {$\mathcal{G}^1_1(\sigma,\rho)=$};
\end{tikzpicture},
$\mathcal H_2^{1, w}(\boldsymbol\sigma)= $ \begin {tikzpicture} [baseline={([yshift={-30pt}]current bounding box.north)},-latex,auto,node distance =1 cm and 1.5cm,on grid,
semithick,
state/.style ={ circle,top color =white, bottom color = white,
draw,black, text=black, minimum width =0.1 cm}]
\node[fill=red,circle] (A)  {$1$};
\node[state] (B) [ right=of A] {$2$};
\node[fill=red,circle] (C) [ right=of B] {$3$};
\node[fill=red,circle] (D) [ right=of C] {$4$};
\node[state] (E) [below=of C] {$5$};
\path (A) edge [bend left =25]  (B);
\path (B) edge [bend left =25]  (C);
\path (C) edge [bend left =25]  (D);
\path (D) edge [bend left =25]  (E);
\path (E) edge [bend left =25]  (A);
%\path (E) edge [bend left =25]  (F);
%\path (F) edge [bend left =25]  (E);

%\node[] at (-1.5,0) {$\mathcal{G}^1_1(\sigma,\rho)=$};
\end{tikzpicture}.
 \end{center}
\end{itemize}

One can notice that for any $s,i$ and $w,$ $\mathcal G_i^{s, w}(\boldsymbol\sigma)$ is a subgraph of $g_{\sigma_i}.$
Fro any permutation $\sigma,$ the graph $g_\sigma$ and all its subgraphs lie in classes of type $\mathcal C_{\boldsymbol\gamma,\boldsymbol\gamma'},$ that we define now.

\begin{itemize}
     
\item Let $\ell, \ell' \in \mathbb N^*,$   $\boldsymbol\gamma:=(\gamma_1, \ldots, \gamma_\ell) \in  (\mathbb N^*)^\ell$ 
and $\boldsymbol\gamma':=(\gamma'_1, \ldots, \gamma'_{\ell'})\in  (\mathbb N^*)^{\ell'}$.
We denote by $\mathcal C_{\boldsymbol\gamma,\boldsymbol\gamma'}$ the set of classes of  graphs  $g_{\boldsymbol\gamma, \boldsymbol\gamma'}$ defined as follows : for $n \ge  \ell + \sum_{i=1}^\ell\gamma_i + \sum_{i=1}^{\ell'} \gamma'_i,$
$g_{\boldsymbol\gamma, \boldsymbol\gamma'} \in \mathbb G^n$ has $\ell+\ell'$ non trivial connected components, for any $j, j' \in \llbracket  \ell+\ell' \rrbracket,$ if $j \neq j'$ then $j$ and $j'$ are in two distinct connected components. For $j \le \ell',$  
the  component containing $j$ is a directed cycle of 
 length $\gamma'_j$ (the cycles of length $1$ being loops). For $\ell'+1 \le j \le \ell+\ell',$ the component containing $j$
has $\gamma_j$ edges and $\gamma_j+1$ vertices : %two vertices of degree 1 and $\gamma_j-1$ vertices of degree 2
 $j$ is  a vertex  of incoming degree $0$ and outgoing degree $1,$ that we call  the \emph{head,} one vertex of incoming degree $1$ and outgoing degree $0$, that we call the \emph{tail}  and $\gamma_j-1$ vertices of incoming degree $1$ and outgoing degree $1.$
 We call such a component \emph{straight}.
 For example, \\
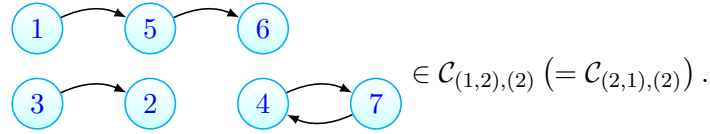
 \begin{center}
 \begin {tikzpicture} [baseline={([yshift={-30pt}]current bounding box.north)},-latex,auto,node distance =1 cm and 1.5cm,on grid,
semithick,
state/.style ={ circle,top color =white, bottom color = processblue!20,
draw,processblue, text=blue, minimum width =0.1 cm}]
\node[state] (C) {$3$};
\node[state] (A) [above right=of C] {$5$};
\node[state] (A2) [ right=of A] {$6$};
\node[state] (B) [above =of C] {$1$};
\node[state] (D) [right =of C] {$2$};
\node[state] (E) [right =of D] {$4$};
\node[state] (F) [right =of E] {$7$};
\path (B) edge [bend left =25]  (A);
\path (A) edge [bend left =25]  (A2);

\path (C) edge [bend left =25]  (D);

\path (E) edge [bend left =25]  (F);
\path (F) edge [bend left =25]  (E);

%\node[] at (-1.5,0) {$\mathcal{G}^1_1(\sigma,\rho)=$};
\end{tikzpicture}
 $\in  \mathcal{C}_{(1,2),(2)}  \left(=  \mathcal{C}_{(2,1),(2)}\right).$ \\ 
 \mbox{}\\
 \end{center}
 Note that the multiplicity of the entries in the vectors $\boldsymbol\gamma$ and $\boldsymbol\gamma'$ is relevant, but their order  is not.

\item By convention, we will extend the previous definition to the case when $\ell' =0,$ that is the case when there is no loop nor cycle. In this case, for $\ell \in \mathbb N^*,$ for $\boldsymbol \gamma:=(\gamma_1, \ldots, \gamma_\ell) \in  (\mathbb N^*)^\ell,$ we denote by $\mathcal T_{\boldsymbol\gamma}$ the class
$\mathcal C_{\boldsymbol\gamma, \emptyset}$ of graphs $g_{\boldsymbol\gamma}$ having   $\ell$ non trivial connected components, all of them being straight,  the component containing $j \le \ell$ having $\gamma_j$ edges and these components being pairwise distinct.

%such that for any $j \in \{1, \ldots, d\},$ the $j$th component has $\gamma_j$ edges and $\gamma_j+1$ vertices  %two vertices of degree 1 and $\gamma_j-1$ vertices of degree 2
 %one vertex of incoming degree $1$ and outgoing degree $0$,  one vertex  of incoming degree $0$ and outgoing degree $1,$ that we call  the head,  and $\gamma_j-1$ vertices of incoming degree $1$ and outgoing degree $1$. We call such a component \emph{straight}. %  and $\mathcal T_\gamma = \cup_{n\in \mathbb N ^*} T_{\gamma}^n.$  
 For example, 
 
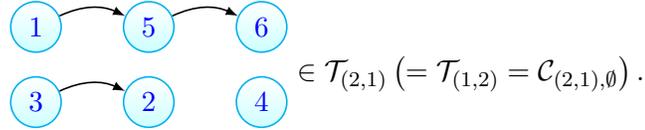
\begin{center}
 \begin {tikzpicture} [baseline={([yshift={-30pt}]current bounding box.north)},-latex,auto,node distance =1 cm and 1.5cm,on grid,
semithick,
state/.style ={ circle,top color =white, bottom color = processblue!20,
draw,processblue, text=blue, minimum width =0.1 cm}]
\node[state] (C) {$3$};
\node[state] (A) [above right=of C] {$5$};
\node[state] (A2) [ right=of A] {$6$};
\node[state] (B) [above =of C] {$1$};
\node[state] (D) [right =of C] {$2$};
\node[state] (E) [right =of D] {$4$};
\path (B) edge [bend left =25]  (A);
\path (A) edge [bend left =25]  (A2);

\path (C) edge [bend left =25]  (D);
%\node[] at (-1.5,0) {$\mathcal{G}^1_1(\sigma,\rho)=$};
\end{tikzpicture}
 $\in  \mathcal{T}_{(2,1)}  \left(=  \mathcal{T}_{(1,2)} = \mathcal{C}_{(2,1), \emptyset}\right).$ %\subset \mathcal{T}_{(2,3)} .$
  %  {\color{orange} probleme allignement à regler }
\end{center}

% {\color {orange} à reformuler et à ajouter un dessin.}
 
%It is important to notice that for any  $\sigma \in \mathfrak S_n$, the graph $g_\sigma$ and all its subgraphs are of class $C_{\boldsymbol\gamma,\boldsymbol\gamma^\prime}$ for some $\boldsymbol\gamma$ and $\boldsymbol\gamma^\prime$. 
 
For a tuple $\boldsymbol \gamma:=(\gamma_1, \ldots, \gamma_\ell) \in  (\mathbb N^*)^\ell,$ and $h \in \mathbb N^*, h\cdot \boldsymbol \gamma  \in  (\mathbb N^*)^{h\ell}$ is the t-tuple obtained by repeating $h$ times the t-uple $\boldsymbol \gamma.$
 
 For example $\mathcal T_{3\cdot (1,4,5)} = \mathcal T_{(1,4,5,1,4,5,1,4,5)} = \mathcal T_{(1,1,1, 4,4,4,5,5,5)}.$\\

\item For any $j \in \llbracket  k \rrbracket,$ we define $\gamma^i(w) := (|\beta_j|)_{\alpha_j=i},$ and $\boldsymbol \gamma(w) := (\gamma^1(w), \ldots, \gamma^k(w)).$

For example, for $w=x_1^4x_2^{-3}x_3^2x_2^5$, $\gamma^1(w)=(4),\gamma^2(w)=(3,5)$ and $\gamma^3(w)=(2).$

\end{itemize}
%Our first step will be to identify the main contributions in the uniform case. 

%To prove this proposition, we will need a technical lemma, that we state below and which is the analogue of Lemma 2.9 in \cite{KaMa20}. We give a general statement for graphs in a class $\mathcal C_{\gamma,\gamma'}$ and then gives the particular form for $\mathcal{T}_\gamma.$

The structure of the graphs and their colored versions follow the following rules :

\begin{lemma} \label{graph_geometry}
Let $w$ be  a word in $F_k$ with reduced form $w= x_{\alpha_q}^{\beta_q}\ldots x_{\alpha_1}^{\beta_1},$ with $q \ge 2$ and $1\leq i\leq k.$ For any $s \subset \llbracket n \rrbracket,$
\begin{enumerate}
    \item A connected component  of ${\mathcal H}_i^{s, w}(\boldsymbol\sigma)$  either is a cycle or is straight.
    \item A connected component  of ${\mathcal H}_i^{s, w}(\boldsymbol\sigma)$ has at least one red vertex.
    \item If a  connected component  of ${\mathcal H}_i^{s, w}(\boldsymbol\sigma)$ has exactly one red vertex, then it is a cycle of length $d$ dividing $\beta_\ell$ for some $\ell$ such that $\alpha_\ell=i$.
    \item If a connected component of ${\mathcal H}_i^{s, w}(\boldsymbol\sigma)$ is straight, then the head and the tail are red. 
    \item If a connected component of ${\mathcal H}_i^{s, w}(\boldsymbol\sigma)$  has exactly two red points, then it is either
    \begin{itemize}
    \item   a straight component of length         $|\beta_\ell|$ for some $\ell$ such that $\alpha_\ell=i,$ 
        \item 
a cycle of length
    $|\beta_\ell|+|\beta_{\ell'}|$ for some $\ell, \ell'$ such that $\alpha_\ell=\alpha_{\ell'}=i,$ 
    \item 
    a cycle of length
    $ j <|\beta_\ell|$ for some $\ell$ such that $\alpha_\ell=i.$
   \end{itemize}
   \item Under the condition $w(\boldsymbol{\sigma})(m)=m$, if $\ell$ is red in ${\mathcal H}_i^{s, w}(\boldsymbol\sigma)$ , then it is red in ${\mathcal H}_j^{s, w}(\boldsymbol\sigma)$  for some $j\neq i$.
\end{enumerate}
\end{lemma} 

\begin{proof} 
Let $w=x_{\alpha_q}^{\beta_q}...x_{\alpha_1}^{\beta_1}$ and $\boldsymbol{\sigma}=(\sigma_1,\dots,\sigma_k).$ For each $m,$ we can describe the sequence $i_0^{m,w}(\boldsymbol \sigma), \ldots, i_r^{m,w}(\boldsymbol \sigma)$ as follows : starting from $m,$ we perform a walk of length $|\beta_1|$ on the graph $g_{\sigma_{\alpha_1}}$, following the oriented edges if $\beta_1 >0$ or going backward if $\beta_1<0,$ then, after $|\beta_1|$ steps, we walk the same way on $g_{\sigma_{\alpha_2}}$ during $|\beta_2|$ steps etc. The colored graph ${\mathcal H}_i^{m, w}(\boldsymbol\sigma)$ is then the part of $g_{\sigma_i}$ explored by this process where each vertex is colored by red if and only it is the entering or exiting point for one of the walks. 
Given this description, the proof of the lemma is straightforward. 
\end{proof}

We gather hereafter several technical bounds that we will need in the proofs of our main results.
For $ \sigma \in \mathfrak S_n$ and $j \in \llbracket n\rrbracket,$ $c_j(\sigma)$ is the length of the cycle of the permutation $\sigma$ containing $j.$

For a graph $g$ of class $\mathcal T_\gamma,$ the probability of the event $\{ \sigma_n \in \mathfrak S_{n,g}\}$ is controlled by the length of the cycle of a given point, say $1.$ This is stated in the following lemma that we recall from \cite{KaMa20}:

 \begin{lemma}  \label{lem:straight1} 
 Let $\boldsymbol\gamma:=(\gamma_1, \ldots, \gamma_\ell) \in  (\mathbb N^*)^\ell$
 and 
$n \ge v:=\ell + \sum_{i=1}^\ell\gamma_i.$
 For any graph $g \in \mathbb G^n$ of class $\mathcal{T}_{\boldsymbol \gamma}$ and  $\sigma_n$ a random permutation with conjugation invariant distribution, we have
\[
  \left(1- \sum_{i=1}^\ell \mathbb{P}(c_1(\sigma_n) \le  \gamma_i) - \frac{\ell-1}{n-1}  \sum_{i=1}^\ell \gamma_i  \right)  \le     \frac{\mathbb{P}(\sigma_n\in\mathfrak{S}_{n,g})(n-\ell)!}{(n-\ell-\sum_{i=1}^\ell \gamma_i)!} \le 1 .
\]

 \end{lemma}

\begin{proof}
This is exactly Lemma 2.9 in \cite{KaMa20}.
\end{proof}

We will also need a control on $\mathbb{P}(\sigma_n\in\mathfrak{S}_{n,g})$ when $g$ is a graph of class $\mathcal C_{\boldsymbol\gamma, \boldsymbol\gamma'},$ with $\ell'>0.$ This will require to look at the length of the cycles of $\ell'$
given points, for example $1, \ldots, \ell',$ provided they lie in pairwise distinct cycles. More precidely, we have:

\begin{lemma}
\label{lem:straight2}
 Let $\boldsymbol\gamma:=(\gamma_1, \ldots, \gamma_\ell) \in  (\mathbb N^*)^\ell,$ 
$\boldsymbol\gamma':=(\gamma'_1, \ldots, \gamma'_{\ell'})\in  (\mathbb N^*)^{\ell'}$
and 
$n \ge v:=\ell + \sum_{i=1}^\ell\gamma_i + \sum_{i=1}^{\ell'} \gamma'_i.$
For $\ell'>0,$ we define 
\[A^{\boldsymbol\gamma'}=\left\{\sigma \in \mathfrak{S}_n :\forall 1\le j\le \ell' : c_j(\sigma)= \gamma'_j   \text { and }, \forall 1\leq j\neq j'\le \ell', j \text{ and } j' \text{ are not in the same cycle}\right\}.
\]

\begin{enumerate}
\item For $\ell'>0, $ if  $g \in \mathbb G^n$ is a graph of class $\mathcal C_{\boldsymbol\gamma, \boldsymbol\gamma'},$ for any random permutation  $\sigma_n$  with conjugation invariant distribution, we have 
\begin{align*}
  &\frac{
  \mathbb{P}(\sigma_n\in\mathfrak{S}_{n,g}){(n-\ell-\ell')!}}{ (n-\ell-\sum_{i=1}^\ell\gamma_i-\sum_{i=1}^{\ell'} \gamma'_i)!} 
  \le \mathbb{P}(\sigma_n\in A^{\boldsymbol\gamma'}).
\end{align*}

Moreover, 
\begin{equation}\label{calculdirect}
\mathbb{P}(\sigma_n\in A^{\gamma'})
 \leq  \mathbb E\left( \prod_{i=1}^{\ell'} \left({\#_{\gamma'_i} \sigma_{n}}\right)\right) \frac{(n-\ell')!}{n!}  \prod_{i=1}^{\ell'} \gamma'_i. 
\end{equation} 
 
\item  If $\rho_n$ is distributed according to the uniform law then, 
\begin{align*}
  &\mathbb{P}(\rho_n\in A^{\boldsymbol\gamma'}) 
  \left(1-\frac{\ell\sum_{i=1}^{\ell'} (\gamma'_i-1)}{n-\ell'}\right) \left(1 -\frac{\ell\sum_{i=1}^{\ell} \gamma_i}{n-\sum{\gamma'_i}}  \right)
  \le\frac{
  \mathbb{P}(\rho_n\in\mathfrak{S}_{n,g}){(n-\ell-\ell')!}}{ (n-\ell-\sum_{i=1}^\ell\gamma_i-\sum_{i=1}^{\ell'} \gamma'_i)!}.
\end{align*}
\end{enumerate}

% In the case $\ell'=0,$ this reads as follows : for any graph $g$ of class $\mathcal{T}_\gamma$, for any random permutation $\sigma_n$ with conjugation invariant distribution,

\end{lemma}

\begin{proof}
%The first point, $\ell'=0,$ is exactly the content of Lemma 2.9  \cite{KaMa20}. 

%We will now give the proofs in the case $\ell'>0.$ 
Let $n \ge v  :=\ell + \sum_{i=1}^\ell\gamma_i + \sum_{i=1}^{\ell'} \gamma'_i$ and $g \in \mathbb G^n$ be a graph of class $ \mathcal C_{\boldsymbol\gamma, \boldsymbol\gamma'}.$
By conjugation invariance, one can assume without loss of generality that
$g=g_{\boldsymbol\gamma, \boldsymbol\gamma'}$ as in the definition of $ \mathcal C_{\boldsymbol\gamma, \boldsymbol\gamma'}$ above and that the non-isolated vertices are $\llbracket  v \rrbracket.$
We denote by  
\[ F:= \{ y=(y_i)_{\ell + \ell'+1 \le i \le v}, y_i \in \{\ell + \ell'+1, \ldots, n\}, \textrm{pairwise distincts}  \}\]
and for any $y \in F,$ we denote by $g_y$ the graph isomorphic to $g$ obtained from $g$ by fixing the vertices $1, \ldots, \ell+\ell'$ and, for any $i \in \{ \ell + \ell'+1, \ldots, v\},$ replacing $i$ by $y_i.$ 

It is easy to check that, for $y, y' \in F,$ if $y \neq y',$ we have $ \mathfrak{S}_{n,g_y} \bigcap \mathfrak{S}_{n,g_{y'}} = \emptyset.$ From there, we easily get the first bound of point 1.:
$$
  \mathbb P(\sigma_n \in \mathfrak{S}_{n,g})  = \frac{ \mathbb P\left(\sigma_n \in \bigcup_{y \in F} \mathfrak{S}_{n,g_y}\right)}{\textrm{card} F} 
  \leq  \frac{\mathbb{P}(\sigma_n\in A^ {\boldsymbol\gamma'})}{\frac{(n-\ell -\ell')!}{(n-v)!}}.
$$

For the second inequality of point 1., we write 

\begin{align*}
\mathbb E\left(\prod_{i=1}^{\ell'} \gamma'_i \left({\#_{\gamma'_i} \sigma_{n}}\right)\right)
=\sum_{(i_1,\dots,i_{\ell'}) \in \llbracket n\rrbracket^{\ell'}} \mathbb E\left(\prod_{j=1}^{\ell'} \mathrm{1}_{\{ c_{i_j}(\sigma_n)=\gamma'_i\}} \right)
&\geq \sum_{(i_1,\dots,i_{\ell'}) \in  \llbracket n\rrbracket^{\ell'} \text{pairwise distinct}} \mathbb E\left( \prod_{j=1}^{\ell'} \mathrm{1}_{\{ c_{i_j}(\sigma_n)=\gamma'_i\}} \right)
\\&= \frac{n!}{(n-\ell')!}\mathbb E\left(\prod_{j=1}^{\ell'} \mathrm{1}_{\{c_{j}(\sigma_n)=\gamma'_i\}} \right) 
\\&\geq \frac{n!}{(n-\ell')!} \mathbb P\left(\sigma_n\in A^{\gamma'}\right).
\end{align*}

%To show the two other bounds, we introduce $\mathfrak{S}_{n}^{\boldsymbol\gamma, \boldsymbol\gamma'} = A^ {\boldsymbol\gamma'} \cap B^{\boldsymbol\gamma,\boldsymbolgamma'},$ with 
%\[ B^{\gamma, \gamma'}:=\bigcap_{i=1}^{\ell} \bigcap_{j=1}^{\gamma_i} \left\{ \sigma \in \mathfrak{S}_{n}, \sigma(i+\ell')^j > \ell+\ell'\right\}\].
 
%Then, by observing that 
%\[  \mathbb P(\sigma_n \in \mathfrak{S}_{n,g} | \mathfrak{S}_{n} \setminus \mathfrak{S}_{n}^{\gamma, \gamma'}) =0,\]
%\[ \mathbb P(\sigma_n \in \mathfrak{S}_{n,g} | \mathfrak{S}_{n}^{\gamma, \gamma'}) = \frac{(n-v)!}{(n-\ell -\ell')!},\]
%we obtain that
%\[ \mathbb P(\sigma_n \in \mathfrak{S}_{n,g} )= \mathbb{P} (\mathfrak{S}_{n}^{\gamma, \gamma'})  \frac{(n-v)!}{(n-\ell -\ell')!}.\]

%When $\ell'=0$,  $\mathfrak{S}_{n}^{\gamma, \gamma'}=B^{\gamma,\gamma'}$ and 
%\begin{align*}
%1-\mathbb{P}(\sigma_n\in B^{\gamma,\gamma'}) &= \mathbb{P}(\bigcup_{i=1}^{\ell} \bigcup_{j=1}^{\gamma_i} \left\{ \sigma_n(i)^j \leq \ell\right\})\\&= \sum_{i=1}^\ell      \mathbb  P(\exists j \le \gamma_i, \sigma_n(i)^j \le \ell) \\&  \le \sum_{i=1}^\ell\mathbb P(\exists j \le \gamma_i, \sigma_n(1)^j \le \ell) \\
%    & = \sum_{i=1}^\ell ( \mathbb P(c_1(\sigma_n) \le \gamma_i) + (\ell - 1) \mathbb P(\exists j \le \gamma_i, \sigma_n(1)^j =2))\\
 %   &\le \sum_{i=1}^\ell  (\mathbb P(c_1(\sigma_n) \le \gamma_i) + (\ell - 1) \frac{\gamma_i}{n-1}). 
%\end{align*}

We now go to the bound in the uniform case  and define :
\[ B^{\boldsymbol\gamma, \boldsymbol\gamma'}:=A^{\boldsymbol\gamma'}\bigcap \left(\bigcap_{i=1}^{\ell} \bigcap_{j=1}^{\gamma_i} \left\{ \sigma \in \mathfrak{S}_{n}, \sigma^j(i+\ell') > \ell+\ell'\right\}\right)\]
\[ C^{\boldsymbol\gamma, \boldsymbol\gamma'}:=A^{\boldsymbol\gamma'}\bigcap \left(\bigcap_{i=1}^{\ell} \bigcap_{j=1}^{\gamma_i} \left\{ \sigma \in \mathfrak{S}_{n}, \sigma^j(i+\ell') > \ell'\right\}\right)
=A^{\boldsymbol\gamma'}\bigcap \left(\bigcap_{i=1}^{\ell'} \bigcap_{j=1}^{\gamma'_i-1} \left\{ \sigma \in \mathfrak{S}_{n}, \sigma^j(i) > \ell+\ell'\right\}\right)
.\]
By observing that for $g=g_{\boldsymbol\gamma, \boldsymbol\gamma'}$,
\[  \mathbb P(\sigma_n \in \mathfrak{S}_{n,g} | \sigma_n \in \mathfrak{S}_{n} \setminus B^{\boldsymbol\gamma, \boldsymbol\gamma'}) =0,\]
and
\[ \mathbb P(\sigma_n \in \mathfrak{S}_{n,g} | \sigma_n \in B^{\boldsymbol\gamma, \boldsymbol\gamma'}) = \frac{(n-v)!}{(n-\ell -\ell')!},\]
we obtain that
\[ \mathbb P(\sigma_n \in \mathfrak{S}_{n,g} )= \mathbb{P} (\sigma_n \in B^{\boldsymbol\gamma, \boldsymbol\gamma'})  \frac{(n-v)!}{(n-\ell -\ell')!}.\]

Moreover, by definition, $ B^{\boldsymbol\gamma, \boldsymbol\gamma'} \subset  C^{\boldsymbol\gamma,\boldsymbol\gamma'} \subset A^{\boldsymbol\gamma'} $ and then
\[\mathbb{P}(\sigma_n\in  B^{\boldsymbol\gamma, \boldsymbol\gamma'})= \mathbb{P}(\sigma_n\in  B^{\boldsymbol\gamma, \boldsymbol\gamma'}|\sigma_n\in  C^{\boldsymbol\gamma, \boldsymbol\gamma'}) 
\mathbb{P}(\sigma_n\in   C^{\boldsymbol\gamma, \boldsymbol\gamma'}|\sigma_n\in A^{\boldsymbol\gamma'}){\mathbb{P}(\sigma_n\in  A^{\boldsymbol\gamma'})}
.\]
For any conjugation invariant permutation, 
\[
\mathbb{P}(\sigma_n\in  C^{\gamma,\gamma'}|\sigma_n\in A^{\gamma'})\geq 1-\sum_{i=1}^{\ell'}\sum_{j=1}^{\gamma'_i-1}\sum_{p=1}^\ell \mathbb{P}\left(\sigma^j_n(i)=p+\ell'|\sigma_n\in A^{\gamma'}\right)= 1-\sum_{i=1}^{\ell'}\frac{\ell(\gamma'_i-1)}{n-\ell'}
.\]
When  $\rho_n$ is a uniform random permutation, then conditionally to $\{\rho_n \in C^{\gamma',\gamma}\}$,
the restriction of $\rho_n$ to $\llbracket n\rrbracket\setminus{\{\sigma^j(i), 1\leq i\leq \ell,1\leq j\leq \gamma^\prime_i\}}$ is uniform. Therefore,
\[
\mathbb{P}(\rho_n\in  B^{\boldsymbol\gamma, \boldsymbol\gamma'}|\rho_n\in  C^{\boldsymbol\gamma, \boldsymbol\gamma'}) \geq 
1-\frac{\ell\sum_{i=1}^{\ell} \gamma_i}{n-\sum{\gamma'_i}}.
\]
This concludes the proof of the lower bound in the uniform case.
\end{proof}
%In fact, under $C^{\gamma',\gamma}$ if we fix the cycles of $1,2,\dots,\ell'$, the law of the remaining entries is that of a uniform permutation of length $n-\sum \gamma'_i$. 
%\[  \mathbb P(\mathfrak{S}_{n}^{\gamma, \gamma'}) =  \mathbb P(B^{\gamma'} |  A^{\gamma, \gamma'}) \mathbb P( A^{\gamma, \gamma'})
%= \mathbb{P}(\forall 1\le j\le \ell' : c_j(\sigma_n)= \gamma'_j ) \mathbb P( A^{\gamma, \gamma'})\]
%and 
%\begin{align*}
%  \mathbb P( A^{\gamma, \gamma'})  & \ge 1 - \sum_{i=1}^{\ell'} \mathbb P(\exists j \le \gamma'_i,  \sigma(i)^j \le \ell+\ell') -  \sum_{i=1}^{\ell} \mathbb P(\exists j \le \gamma_i,  \sigma(i + \ell')^j \le \ell+\ell')
% \end{align*}
% and, for any $1 \le i \le \ell',$

\section{Asymptotics of fixed points}

For readability reasons, before proving the general statements, we will prove the convergence of the number of fixed points under weaker assumptions. The statement is the following :
\begin{theorem}\label{th:general}
Let  $w$ be a word in  $F_k$ with canonical form $w= x_{k_r}^{\varepsilon_r}\ldots x_{k_1}^{\varepsilon_1},$  where $k_1 \neq k_r$ (there is at least two different letters). For $j\le k,$ we denote by $r_j:=\mathrm{card}\{i, k_i=j\}.$ We write $w= \Omega^d,$ with   $\Omega$ which is not a power and $d \ge 1.$  

Let $\sigma_{1,n}, \ldots, \sigma_{k,n}$ be independent random permutations, with conjugation invariant distributions satisfying the following assumptions:
\begin{equation}\label{hyp:main}
\forall j \in \llbracket k\rrbracket, \forall p \in \mathbb N^*, c_1, \ldots, c_p \in   \widehat{\mathcal N}_j(w)   \textrm{ such that } \sum_{i=1}^p c_i\le r_j,
\lim_{n\to\infty} \mathbb E\left( \prod_{i=1}^p \left(\frac{\#_{c_i} \sigma_{j,n}}{\sqrt{n}}\right)\right) =0
\end{equation}
and
\begin{equation}\label{hyp:main_2}
\forall j \in \llbracket k\rrbracket, \forall i \in \mathcal N_j(w), \lim_{n\to\infty} \mathbb E\left(\frac{\#_{i} \sigma_{j,n}}{{n}}\right) =0.
\end{equation}
Then we have that
\[\mathbb E\left(\#_1 w(\sigma_{1,n}, \ldots, \sigma_{k,n})\right) \xrightarrow[n \rightarrow \infty]{} \psi(d), \]
where we recall that $\psi(d)$ is the numbers of divisors of the integer $d.$
\end{theorem}

In particular, in the case when $d=1,$ Theorem \ref{th:general} reads :

\begin{corollary} \label{corollary}

Let  $\Omega\in F_k$ with canonical form $\Omega= x_{k_r}^{\varepsilon_r}\ldots x_{k_1}^{\varepsilon_1} \neq 1 $ with $k_1 \neq k_r,$  and which  is not a power.
% Let $\Omega$ be a word which is not a power and $d \ge 1$  such that $w=\Omega^d.$ 

Under Assumptions \eqref{hyp:main} and \eqref{hyp:main_2}, we have that
\[\mathbb E\left(\#_1 \Omega(\sigma_{1,n}, \ldots, \sigma_{k,n})\right) \xrightarrow[n \rightarrow \infty]{}  1. \]

\end{corollary}

We start with the identification of the graphs that are contributing to the limit in the uniform case :

\begin{proposition} \label{prop:uniform}
Let $w = \Omega^d$ be a word  such that  $\Omega$ is not a power. We write the reduced form of $\Omega$ as $x_{\alpha_\ell}^{\beta_\ell}\ldots  x_{\alpha_1}^{\beta_1},$ with $\alpha_\ell \neq \alpha_1.$     Assume that $\rho_{1,n}, \ldots, \rho_{k,n}$ are independent and \textbf{uniformly distributed} on $\mathfrak S_n.$ Then,  for any $\hat{g}_1,\hat{g}_2,\dots \hat{g}_k \in  \widehat{\mathbb{G}},$
\[\lim_{n\to\infty} n\mathbb{P}\left( \forall i \in \llbracket  k \rrbracket, \widehat{\mathcal{G}}^{1,w}_{i}(\rho_{1,n}, \ldots, \rho_{k,n})= \hat{g}_i,  c_1(w(\rho_{1,n}, \ldots, \rho_{k,n}))=1 \right)= \sum_{h| d}\prod_{i=1}^k  \mathbbm{1}_{\{\hat{g}_i=\mathcal {T}_{h\cdot\gamma^i(\Omega)}\}}, \]
where the sum runs over all divisors $h$ of $d$ and $ \forall i \in \llbracket  k \rrbracket, $  $\gamma^i(\Omega) = (|\beta_j|)_{\alpha_j=i}.$ 
\end{proposition}

%We now go to the proof of Proposition \ref{prop:uniform}.
\begin{proof} 
We first recall that for any random permutation $\tau_n,$
\[\mathbb{E}(\#_1(\tau_n))=\mathbb{E}\left(\sum_{i=1}^n  \mathbbm{1}_{\{c_i(\tau_n)=1\}}\right)=\sum_{i=1}^n\mathbb{E}(\mathbbm{1}_{\{c_i(\tau_n)=1\}})=\sum_{i=1}^n\mathbb{P}({c_i(\tau_n)=1})\]
and if $\tau_n$ is conjugation invariant then $\mathbb{P}({c_i(\tau_n)=1})=\mathbb{P}({c_1(\tau_n)=1})$
 and 
\[\mathbb{E}(\#_1(\tau_n))=n\mathbb{P}({c_1(\tau_n)=1}).\]
 
 According to Proposition \ref{prop:pp},  for $\rho_{1,n}, \ldots, \rho_{k,n}$ independent and uniformly distributed,  
 $$\lim_{n \rightarrow \infty}\mathbb{E}(\#_1 w(\rho_{1,n}, \ldots, \rho_{k,n})) = \psi(d).$$ This is equivalent to the following: 
\[ \lim_{n\to\infty} \sum_{ (\hat{g}_i)_{1\le i \le  k}  } n\mathbb{P}\left( \forall i \in \llbracket  k \rrbracket, \widehat{\mathcal{G}}^{1,w}_{i}(\rho_{1,n}, \ldots, \rho_{k,n})= \hat{g}_i , c_1(w(\rho_{1,n}, \ldots, \rho_{k,n}))=1 \right)=\psi(d). \]

Our candidates for being the main contributions are the $k$-tuples of classes $(\mathcal {T}_{h \cdot\gamma^1(\Omega)}, \ldots, \mathcal {T}_{h \cdot\gamma^k(\Omega)})$ with $\gamma^i(\Omega) = (|\beta_j|)_{\alpha_j=i}$  and $h|d.$ As there are $\psi(d)$ such $k$-tuples,  one  only needs to show that 
\[\lim_{n\to\infty} n\mathbb{P}\left( \forall i \in \llbracket  k \rrbracket, \widehat{\mathcal{G}}^{1,w}_{i}(\rho_{1,n}, \ldots, \rho_{k,n})= \mathcal {T}_{h \cdot\gamma^i(\Omega)} , c_1(w(\rho_{1,n}, \ldots, \rho_{k,n}))=1 \right) \ge 1.\]

We denote by $t:=\sum_{i=1}^\ell|\beta_i|$  the total length of the word $\Omega$ and $r=td$ the length of $w$. Let $h|d$ and, for $n$ large enough,  let $j_1,j_2, \ldots, j_{ht-1}$ be $ht-1$ indices in $\{2, \ldots, n\}$ which are two by two distinct.

%{\color{red}  Ancienne preuve : 

%Then there exists  at least $(t-1)!$ $k$-tuple of graphs
%$(g_1, \ldots, g_k)$, whose non-isolated vertices belong to $\{1, j_2, \ldots, j_{t-1}\}$ such that for any  $i \in \llbracket  k \rrbracket,$ $\hat g_i = \mathcal {T}_{h \cdot \gamma^i(\Omega)}, $ and the vertex 1 appears at least once as the head of a component and at least once as the tail of another component.}

%{\color{orange} 
%Nouvelle preuve, 

Let $\boldsymbol \sigma := (\sigma_1, \ldots, \sigma_k)$ be such that

%Let   $(g_1,\dots,g_k)$  be equal to  $(\mathcal G_1^{1, w}(\boldsymbol\sigma),\dots,(\mathcal G_k^{1, w}(\boldsymbol\sigma))$ when
 \[ i_0^{1, w}(\boldsymbol\sigma)=1,  i_1^{1, w}(\boldsymbol\sigma)= j_1 , , \ldots, i_{ht-1}^{1, w}(\boldsymbol\sigma)=j_{ht-1}, i_{ht}^{1, w}(\boldsymbol\sigma)=1.\]
By construction, $\forall 1\leq  i\leq k$, 
$ \widehat{\mathcal{G}^{1,w}_{i}}(\boldsymbol\sigma)= \mathcal {T}_{h \cdot\gamma^i(\Omega)}.$
%}
To simplify the notations, for any $i \in \llbracket k\rrbracket,$ we denote by $g_i:= \mathcal{G}^{1,w}_{i}(\boldsymbol\sigma).$
By Lemma \ref{lem:straight1} and independence, we have the following lower bound :
\begin{multline*}
 \mathbb P\left( \forall i  \in \llbracket k\rrbracket, \rho_{i,n} \in  \mathfrak{S}_{n,g_i}\right) \\ \ge  \prod_{i=1}^k \frac{(n-h\ell_i-h\sum_{j=1}^{\ell_i} \gamma_j^i(\Omega))!}{(n-h\ell_i)!}  \left(1- h\sum_{j=1}^{\ell_i} \mathbb{P}(c_1(\rho_{n,i}) \le  \gamma_j^i(\Omega)) - \frac{h\ell_i-1}{n-1}  h\sum_{j=1}^{\ell_i} \gamma_j^i(\Omega) \right),  
\end{multline*}
where $\ell_i$ is the number of non-trivial connected components of $\gamma^i(\Omega).$

From there, we get that 
\begin{multline*}
    \mathbb{P}\left( \forall i \in \llbracket  k \rrbracket, \widehat{\mathcal{G}}^{1,w}_{i}(\rho_{1,n}, \ldots, \rho_{k,n})= \mathcal {T}_{h\cdot\gamma^i(\Omega)} , c_1(w(\rho_{1,n}, \ldots, \rho_{k,n}))=1 \right) \\ 
    \ge \prod_{i=1}^k \frac{(n-h\ell_i-h\sum_{j=1}^{\ell_i} \gamma_j^i(\Omega))!}{(n-h\ell_i)!}  \left(1- \sum_{j=1}^{\ell_i} \mathbb{P}(c_1(\rho_{n,i}) \le  \gamma_j^i(\Omega)) - \frac{h\ell_i-1}{n-1}  h\sum_{j=1}^{\ell_i} \gamma_j^i(\Omega) \right)  \frac{(n-1)!}{(n-ht)!} 
\end{multline*}

Since, for any $i \in \llbracket k\rrbracket,$  $\sum_{j=1}^{\ell_i} \mathbb{P}(c_1(\rho_{i,n}) \le  \gamma_j^i(\Omega)) + \frac{h\ell_i-1}{n-1} h \sum_{j=1}^{\ell_i} \gamma_j^i(\Omega) \to 0,$  and we get that
%By construction, the total number of edges of $(g_1, \ldots, g_k)$ is $\sum_{}= r$
\begin{multline*}
 \liminf_{n\to\infty} n\mathbb{P}\left( \forall i \in \llbracket  k \rrbracket, \widehat{\mathcal{G}}^{1,w}_{i}(\rho_{1,n}, \ldots, \rho_{k,n})= \mathcal {T}_{h\cdot\gamma^i(\Omega)}, c_1(w(\rho_{1,n}, \ldots, \rho_{k,n}))=1 \right)\\
 \geq \liminf_{n\to\infty}  n \frac{1}{n^{h\sum_{i=1}^k\sum_{j=1}^{\ell_i} \gamma_j^i(\Omega) }}n^{ht-1} \geq 1,
\end{multline*}
where the last inequality comes from the fact that 
$ \sum_{i=1}^k\sum_{j=1}^{\ell_i} \gamma_j^i(\Omega) = t.$
This concludes the proof.

\end{proof}

We are now ready for the proof of Theorem \ref{th:general}.

\begin{proof}
It is enough to show that, under the assumptions \eqref{hyp:main} and \eqref{hyp:main_2}, 
 for any $\hat{g}_1,\hat{g}_2,\dots \hat{g}_k \in  \widehat{\mathbb{G}},$
\[\lim_{n\to\infty} n\mathbb{P}\left( \forall i \in \llbracket  k \rrbracket, \widehat{\mathcal{G}}^{1, w}_{i}(\sigma_{1,n}, \ldots, \sigma_{k,n})= \hat{g}_i , c_1(w(\sigma_{1,n}, \ldots, \sigma_{k,n}))=1 \right)=\sum_{h|d}\prod_{i=1}^k \mathbbm{1}_{\{\hat{g}_i=\mathcal {T}_{h\cdot\gamma^i(\Omega)}\}}. \]

Let  $g$ be  of class $\mathcal T_{\boldsymbol\gamma},$ with $\boldsymbol\gamma=(\gamma_1,\gamma_2,\dots,\gamma_\ell).$ 
Then, by  Lemma \ref{lem:straight1}, for any random permutation $\sigma_n$ with conjugation invariant distribution on $\mathfrak{S}_n$ and $\rho_n$ uniformly distributed   
we have
\begin{equation}\label{eq:straightmaj}
\mathbb P(\sigma_n \in \mathfrak{S}_{n,g}) \le \frac{1}{1-\frac{1}{n}\sum_{i=1}^\ell {\gamma_i} - \frac{\ell-1}{n-1}\sum_{i=1}^\ell {\gamma_i}  }\mathbb P(\rho_n \in \mathfrak{S}_{n,g}).
\end{equation}
Therefore, for any $\varepsilon >0,$ 
 there exists $n_0$ such that 
for any $n\geq n_0$,

\begin{equation}
\mathbb P(\sigma_n \in \mathfrak{S}_{n,g}) \le  (1+\varepsilon) \mathbb P(\rho_n \in \mathfrak{S}_{n,g})  
.\end{equation}
Similarly, for any random permutation $\sigma_n$ with conjugation invariant distribution on $\mathfrak{S}_n$ satisfying Assumption \eqref{hyp:main_2}, we have 
\begin{equation} \label{eq:straightmin}
\mathbb P(\sigma_n \in \mathfrak{S}_{n,g}) \ge \left(1-\sum_{i=1}^\ell \mathbb P(c_1(\sigma_n)=i) - \frac{\ell-1}{n-1}\sum_{i=1}^\ell {\gamma_i}  \right) \mathbb P(\rho_n \in \mathfrak{S}_{n,g})\ge  (1-\varepsilon) \mathbb P(\rho_n \in \mathfrak{S}_{n,g}).
\end{equation}

Therefore, if for any  $i \in \llbracket k\rrbracket,$ there exists $\gamma^i$ such that $\hat g_i =  \mathcal T_{\gamma^i},$
 from Proposition \ref{prop:uniform}, we get that
 \[\lim_{n\to\infty} n\mathbb{P}\left( \forall i \in \llbracket  k \rrbracket, \widehat{\mathcal{G}}^{1, w}_{i}(\sigma_{1,n}, \ldots, \sigma_{k,n})= \hat{g}_i , c_1(w(\sigma_{1,n}, \ldots, \sigma_{k,n}))=1 \right)=\sum_{h|d}\prod_{i=1}^k \mathbbm{1}_{\{\hat{g}_i=\mathcal {T}_{h \cdot\gamma^i(w)}\}}. \]

 We now want  to show that graphs containing loops do not contribute. Let $g$ be a fixed (non-colored) graph. If $h$ is a colored graph such that $\dot h = g,$ then $\hat h$ belongs to a finite set of classes of colored graphs, with cardinal independent of $n.$
 Therefore, if we fix some classes of colored graphs   $(\hat{h}_i)_{1\leq i\leq k}$ such that one of the components is not straight, it is enough to show that
 
 \begin{align} \label{colored_limit}
 \lim_{n\to\infty} n\mathbb{P}\left( \forall i \in \llbracket  k \rrbracket, \widehat{\mathcal{H}}^{1, w}_{i}(\sigma_{1,n}, \ldots, \sigma_{k,n})= \hat{h}_i , c_1(w(\sigma_{1,n}, \ldots, \sigma_{k,n}))=1 \right) =0 .    
 \end{align}
  
 One can  assume that there exists $(\sigma_1,\dots,\sigma_k)$ such that  $c_1(w(\sigma_{1,n}, \ldots, \sigma_{k,n}))=1$ and  $\forall i \in \llbracket  k \rrbracket,$
 $\widehat{\mathcal{H}}^{1, w}_{i}(\sigma_{1}, \ldots, \sigma_{k})= \hat{h}_i$, otherwise \eqref{colored_limit} is trivial. 
 To simplify the notations, we denote in the sequel $h_i:=\mathcal{H}^{1, w}_{i}(\sigma_{1}, \ldots, \sigma_{k}),$ for any $i \in \llbracket  k \rrbracket.$
 
Let $\widetilde{h}_i$ be the same graph as $h_i$ after removing one edge from each cycle having at least two red vertices.
In particular, $\widetilde{h}_i$ and ${h}_i$ have the same set of non-trivial vertices. We denote by $\mathcal{C}^{\gamma^i,{\gamma'}^i}$  the class of $\dot{\widetilde{h}}_i,$ $\ell_i$ and $\ell'_i$ being the number of non-zero components of $\gamma^i$ and $\gamma'^i$ respectively.
Let $\mathcal V_i$ be the set of non-trivial vertices of $h_i$ and $r_i$ is the number of red vertices of $h_i.$

A direct consequence of Lemma~\ref{graph_geometry} is that

\begin{align*}
    \mathrm{card}\left(\bigcup_{i=1}^k \mathcal V_i\right) \leq      \sum_{i=1}^k\left( \mathrm{card}(\mathcal V_i) - r_i/2\right)  \leq \sum_{i=1}^k \sum_{j=1}^{\ell_i} \gamma^i_j+\sum_{j=1}^{\ell'_i}\left({\gamma_j'}^i-1/2\right)=:deg.
\end{align*}
%where $r_j^i$ is the number of red points in the $j$-th cycle of $\widetilde{h}_i.$

  There is at most $O(n^{deg-1})$ possible $k$-tuples of graphs $(h'_i,\dots,h'_k)$ such that $\forall i\in \llbracket  k \rrbracket,$ $\widehat{h'_i}=\widehat{h_i}$ and  there exists $(\sigma_1,\dots,\sigma_k)$ satisfying 
$\forall i \in \llbracket  k \rrbracket,$
 $\widehat{\mathcal{H}}^{1, w}_{i}(\sigma_{1}, \ldots, \sigma_{k})= \hat{h'_i}$ and  $c_1(w(\sigma_{1}, \ldots, \sigma_{k}))=1$. Indeed, one need to choose at most $deg-1$ non-trivial vertices in $\llbracket2,n \rrbracket$ ($1$ is a non-trivial vertex), and then choose their positions (not all choices are admissible).

 We have 
 \begin{align*} 
 %\begin{equation}
 \mathbb{P}\Big( \forall i \in \llbracket  k \rrbracket, \widehat{\mathcal{H}}^{1, w}_{i}(\sigma_{1,n}, \ldots, \sigma_{k,n})= \hat{h}_i & ,  c_1(w(\sigma_{1,n}, \ldots, \sigma_{k,n}))=1 \Big) \\
  &\leq O(n^{deg-1})\mathbb{P}\left(\forall i \in \llbracket k\rrbracket, \mathcal{H}^{1, w}_{i}(\sigma_{1,n}, \ldots, \sigma_{k,n})= {h}_i \right) \\
  &\leq O(n^{deg-1})  \mathbb{P}\left(\forall i \in \llbracket k\rrbracket, \sigma_{i,n} \in  \mathfrak S_{n,{\dot{h}}_i} \right) \\
 &\leq O(n^{deg-1})  \prod_{i=1}^k\mathbb{P}\left(\sigma_{i,n} \in  \mathfrak S_{n,{\dot{h}}_i} \right)
 \end{align*}
 
 By Lemma \ref{lem:straight2}, we have 
% \begin{align*}
\begin{equation}\label{eq:contruct}
    \mathbb{P}{\left( \sigma_{i,n} \in   \mathfrak S_{n,{\dot{h}}_i} \right)} \leq  \mathbb{P}\left( \sigma_{i,n} \in   \mathfrak S_{n,\dot{\widetilde{h}}_i} \right)= O\left(n^{-\sum_{j=1}^{\ell_i}\gamma^i_j-\sum_{j=1}^{\ell'_i}{\gamma'_j}^i + \ell'_i}\right) \mathbb{P}(\sigma_{1,n}\in A^{\boldsymbol{\gamma'}^i}).
\end{equation}

% \end{align*}

Moreover, when $\ell'_i \neq 0,$ by Lemma~\ref{graph_geometry}, for any $j \le \ell'_i$,  $\gamma'^i_j | \beta_\ell$ for some $\ell$ such that $\alpha_\ell=i.$ Consequently,
using \eqref{calculdirect} and Assumption \eqref{hyp:main},  
 \begin{align*}
     \mathbb{P}(\sigma_{i,n}\in A^{{\gamma'}^i}) \leq o(n^{\frac{\ell'_i}{2}}) O(n^{-\ell'_i})= o(n^{-\frac{\ell'_i}{2}})
 \end{align*}
 Therefore, if there exists $i \le k$ such that $\ell'_i \neq 0,$ then \eqref{colored_limit} holds.

We now consider the case when $\forall i \le k,$ $\ell'_i =0.$ It means that every cycle of the $h_i$'s has at least two red vertices. In this case, by Lemma~\ref{graph_geometry}, we have 
\begin{equation*}
    \mathrm{card}\left(\bigcup_{i=1}^k \mathcal V_i\right) \leq      \sum_{i=1}^k\left( \mathrm{card}(\mathcal V_i) - r_i/2\right)  \leq \sum_{i=1}^k \sum_{j=1}^{\ell_i} \gamma^i_j = deg.
\end{equation*}
The second inequality is strict as soon as one component has at least three red vertices.
In this case, we trivially  bound $  \mathbb{P}(\sigma_{i,n}\in A^{{\gamma'}^i})$ by 1 in 
\eqref{eq:contruct} to conclude the proof.

If all the components have at most two red vertices, then, by Lemma ~\ref{graph_geometry}, the components are straight or a cycle of length $|\beta_\ell|+|\beta_{\ell'}|$ for some $\ell$ and $\ell'$ or a cycle of length $j < |\beta_\ell|$ for some $\ell$ such that $\alpha_\ell =i,$ for some $1\le i \le k.$
 In this case, one can conclude similarly to the preceding case, using Assumption \eqref{hyp:main} to conclude.

\end{proof}

\section{ Proof of Theorem  \ref{th:conv_dis} and Corollary \ref{corollary_dis}}

\label{sec:proof2}

Let us first introduce some more notations. Let $\lambda = (\lambda_1, \ldots, \lambda_m)$ be a Young diagram, that is a sequence of nonnegative integers in decreasing order $\lambda_1 \geq \ldots \geq \lambda_m \geq 0.$ It is customary to represent them with $\lambda_i$ empty boxes on row number $i.$ 
For two Young diagrams $\lambda, \mu,$ we say that $\mu \subset \lambda$ if $\mu$ can be obtained from $\lambda$ by removing some rows of $\lambda.$
We denote by $\ell(\lambda)$ the number of non-empty rows of $\lambda$ and $|\lambda| = \sum_{i=1}^{\ell(\lambda)}\lambda_i$ the number of boxes of the diagram.

For example, $\mu = (3,1) = {\tiny\yng(3,1)} \subset  {\tiny\yng(3,3,1)} = (3,3,1)= \lambda $

For any pair $(\lambda,  \mu)$ of Young diagrams such that  $\mu \subset \lambda,$ for any $n \ge \ell(\lambda),$ an admissible filling of type $(\lambda, \mu, n)$ is a filling of $\mu$ such that 
\begin{itemize}
\item all entries are in $\llbracket n \rrbracket$ and are pairwise distinct,
    \item the entries in the first column are increasing, 
    \item in every row, the first entry is the smallest entry,
    \item $\forall i \le \ell(\lambda),$ $i$ is in a row of $\mu$ of length $\lambda_i.$
\end{itemize}

\noindent
Let $K_{\lambda,\mu,n}$ be the number of admissible fillings of type $(\lambda,\mu,n)$. For example, 
\begin{itemize}
    \item $K_{(3,1),(1),n}=0$,  because $1$ should be in a row of $\mu$ of length 3.
    \item $K_{(2,2,2,1),(2,1),n}=0$,  because 1, 2 and 3 should go in the first row of $\mu,$ which has only two boxes.
    \item $K_{(3,3,1),(3,1),n}=2 (n-3) $ because the only admissible fillings in this case are of the form  
    $\young(12j,3)$ and  $\young(1j2,3)$ for some $4\le j\le n$. 
\end{itemize}

An important remark is that, for any Young diagram $\lambda,$ we obviously have $\lambda \subset \lambda$ and, for any $n \ge |\lambda|,$
 $$K_{\lambda,\lambda,n}=\frac{(n-\ell({\lambda}))!}{(n-| \lambda|)!}, $$
because $1, 2 , \ldots , \ell({\lambda})$ should go in the first position of each row of $\lambda$ and we then complete freely the $|\lambda | - \ell({\lambda})$ remaining boxes
with distinct numbers between $\ell({\lambda})+1$ and  $n.$ 
By a similar argument, for any $\mu\subset\lambda$, there exists $C_{\lambda,\mu}$ such that for any $n\geq |\lambda|$,
\begin{equation}\label{fillingnb}
    K_{\lambda,\mu,n}=C_{\lambda,\mu}
\frac{(n-\ell({\lambda}))!}{(n-| \mu|)!}
=C_{\lambda,\mu}n^{|\mu|-\ell(\lambda)}(1+o(1)).
\end{equation}

Let us now go to the proof of   Theorem  \ref{th:conv_dis} itself. The scheme is similar to the proof of Theorem \ref{th:general}, except we have to follow simultaneously 
the trajectory of several starting points $1, 2, \ldots, m.$ As we will see, a crucial point is that it is enough to consider the case when they lie in different cycles.
We introduce the following event : for a Young diagram $\lambda,$ $j \le \ell(\lambda),$ $k_1, \ldots, k_j$ distinct integers in $\llbracket  \ell(\lambda)\rrbracket,$ we denote by
\begin{multline*}
A_{k_1, \ldots, k_j}^{\lambda,\Omega} := \left\{\boldsymbol\sigma \in (\mathfrak S_n)^k, \forall i \le j,  c_{k_i}(\Omega(\boldsymbol\sigma))= \lambda_{k_i} \right. \\
\left. \textrm{ and } 
 \forall i \neq j, k_i \textrm{ and } k_j \textrm{ lie in pairwise distinct cycles of  } \Omega(\boldsymbol\sigma)\right\}.
 \end{multline*}

It is an extension of the event $A^{\gamma'}$ defined in Lemma \ref{lem:straight2}, in the sense that if $\lambda$ is a Young diagram with rows $\gamma^\prime_1, \ldots, \gamma^\prime_\ell$ in decreasing order then 
$A_{1, \ldots, \ell(\lambda)}^{\lambda,x_1} = A^{\gamma'}.$\\

The first key step is the following lemma :

\begin{lemma} \label{Expectation-probability-conversion}
Let $\sigma_n$ and $\rho_n$ be two conjugation invariant random permutations in $\mathfrak{S}_n.$  Then the following two properties are equivalent :

\begin{enumerate}
    \item 
For any Young diagram $\mu$, $ \lim_{n\to\infty} n^{\ell(\mu)}  \mathbb P\left( \sigma_n \in  A_{1, \ldots, \ell(\mu)}^{\mu,x_1}  \right) = \lim_{n\to\infty} n^{\ell(\mu)} \mathbb{P}\left(\rho_n \in  A_{1, \ldots, \ell(\mu)}^{\mu,x_1}   \right) .$

\item  For any $P\in \mathbb{R}[x_1,x_2,\dots,x_{d'}],$ monomial, $$\lim_{n \rightarrow \infty } \mathbb E\left( P(\#_1 \sigma_n, \ldots, \#_{d'} \sigma_n)\right) = \lim_{n \rightarrow \infty } \mathbb E\left( P(\#_1 \rho_n, \ldots, \#_{d'} \rho_n)\right).$$ 

\end{enumerate}

\end{lemma}
\begin{proof}
For any $m \in \mathbb N^*,$ $p_1, \ldots, p_m \in \mathbb N,$
\begin{align*}
    a_{p_1, \ldots, p_m}^n &:= \mathbb E \left( \left(\sum_{i=1}^n \mathbf 1_{\{c_i(\sigma_n)=1\}}\right)^{p_1}\ldots \left(\sum_{i=1}^n \mathbf 1_{\{c_i(\sigma_n)=m\}}\right)^{p_m}\right) \\
    & = \sum_{\tiny \begin{array}{rcl}
   1 \le i_1^1,&\ldots,&i_{p_1}^1 \\
     i_1^m,&\ldots,&i_{p_m}^m\le n
     \end{array}
    } \mathbb P\left( \forall j \le m, \forall s \le p_j, c_{i_s^j}(\sigma_n)=j\right).
\end{align*}
We denote by $\lambda$  the Young diagram $(m,\dots,m,\dots,1,\dots,1)$ where $j$ appears  $p_j$ times. 
For any $j \le m,$ let us denote by $t_j := \textrm{card}\{i_1^j,  \ldots, i_{p_j}^j\} \le p_j$ the number of distinct indices among $i_1^j,  \ldots, i_{p_j}^j$ and $\pi$ the Young tableau $(m,\dots,m,\dots,1,\dots,1)$ where $j$ appears $t_j$ times so that $\ell(\pi)=t_1+\dots+t_m.$. Note that $\pi\subset\lambda.$

Then, by conjugation invariance of the distribution, we have, 
\[ \mathbb P\left(\forall j \le m, \forall s \le p_j, c_{i_s^j}(\sigma_n)=j\right) = \mathbb P\left(\forall j \le m, \forall s \textrm{ such that } 1+ \sum_{ \ell=1}^{j-1} t_\ell \le s \le \sum_{ \ell=1}^{j} t_\ell , c_s(\sigma_n)=j\right),\]
so that
\begin{align*}
  a_{p_1, \ldots, p_m}^n & = \sum_{\pi\subset\lambda}n^{ t_1+ \ldots+t_m}(1+o(1))\left(\prod_{j=1}^m \mathcal{P}(p_j,t_j)\right) \mathbb P\left( \forall i \le  t_1+ \ldots+t_m, c_i(\sigma_n)=\pi_i\right)\\
  & = \sum_{\pi\subset\lambda} n^{ \ell(\pi)}(1+o(1)) \left(\prod_{j=1}^m \mathcal{P}(p_j,t_j)\right)\mathbb P\left( \forall i \le  \ell(\pi), c_i(\sigma_n)=\pi_i\right),
\end{align*}
where $ \mathcal{P}(p,t)$ is the number of partitions of $p$ with $t$ parts.\\

%The last step is to show that it is enough to consider the case when $1, 2, \ldots, m$ are in distinct cycles. Indeed, for any Young diagram $\pi$,

We now denote by $B_\pi=\{\sigma\in\mathfrak{S}_n,\forall i \leq \ell(\pi), c_i(\sigma)=\pi_i\}.$

$\forall\sigma\in B_\pi$, we define  $f_\pi(\sigma)$ as follows : 
let $\pi_\sigma$ be the Young diagram following the cycle structure of $\sigma$ restricted to the cycles having at least one element in $\llbracket \ell(\pi)\rrbracket$. Then $f_\pi(\sigma)$ is the filling  of $\pi_\sigma$ with the elements of the cycles of $\sigma$ starting with the smallest element in each cycle and so that the first column is increasing. 

For example, for $\sigma= (1,7,8)(9,3,2)(4,6)(10,13,6)(12,11)$ and $\pi=(3,3,3,2)$,
$f_\pi(\sigma)=\young(178,293,46).$ \\

Note that for $n$ large enough, $f_\pi(\sigma)$ is necessarily an admissible filling  of type $(\pi,\pi_\sigma,n)$. We have then, 
$$\mathbb{P}(\sigma_n\in B_\pi)=\sum_{\mu\subset\pi} \sum_{f \text{ admissible filling of type } (\pi,\mu,n)}\mathbb{P}(f_\pi(\sigma_n)=f).$$ 
By conjugation invariance, for any admissible filling of type $(\pi,\mu,n)$,

$$\mathbb{P}(f_\pi(\sigma_n)=f) \frac{(n-\ell(\mu))!}{(n-|\mu|)!}=\mathbb{P}(\sigma_n\in A_{1,\dots,\ell(\mu)}^{\mu,x_1})$$

%\[  \mathbb P\left( \forall i \le  \ell(\lambda), c_i(\sigma_n)=\lambda_i |  \sigma_n\in A_{k_1, \ldots, k_J}^{\mu,x_1} \right)  = \frac{K_{\lambda, \mu,n}}{K_{\mu, \mu,n}}.\]

%\begin{align*}
 % \mathbb P\left( \forall i \le  \ell(\lambda), c_i(\sigma_n)=\lambda_i\right) & = \sum_{\{k_1, \ldots, k_J\} \subset \llbracket  \ell(\lambda)\rrbracket }   P\left( \forall i \le  \ell(\lambda), c_i(\sigma_n)=\lambda_i \textrm{ and } \forall i \neq j, k_i  \textrm{ and } k_j \textrm{ are in different cycles of }  \sigma_n\right) \\
 % & =  \sum_{\{k_1, \ldots, k_J\} \subset \llbracket  \ell(\lambda)\rrbracket }  \mathbb P\left( \forall i \le  \ell(\lambda), c_i(\sigma_n)=\lambda_i |  \sigma_n \in A_{k_1, \ldots, k_J}^{\lambda,x_1} \right)  \mathbb P\left(\sigma_n\in A_{k_1, \ldots, k_J}^{\lambda,x_1} \right).
%\end{align*}
%Now, by conjugation invariance,
%\[  \mathbb P\left( \sigma_n\in A_{k_1, \ldots, k_J}^{\lambda,x_1} \right) =   \mathbb P\left( \sigma_n\in A_{1, \ldots, \ell(\mu)}^{\mu,x_1} \right),\]
so that 
\[  \mathbb P\left( \forall i \le  \ell(\pi), c_i(\sigma_n)=\pi_i\right) = \sum_{\mu \subset \pi} \frac{K_{\pi, \mu,n}}{K_{\mu, \mu,n}} \mathbb P\left(\sigma_n\in A_{1, \ldots, \ell( \mu)}^{\mu,x_1} \right).\]

\begin{align*}
  a_{p_1, \ldots, p_m}^n &= (1+o(1))
 \left (  n^{ \ell(\lambda)}\mathbb P\left(\sigma_n\in A_{1, \ldots, \ell( \lambda)}^{\lambda,x_1}\right) +\sum_{\mu \subset \lambda, \lambda\neq\mu} C_{\lambda_\mu}\frac{n^{ \ell(\lambda)}K_{\lambda, \mu,n}}{K_{\mu, \mu,n}} \mathbb P\left(\sigma_n\in A_{1, \ldots, \ell( \mu)}^{\mu,x_1} \right)\right)
 \\&= (1+o(1))
 \left (  n^{ \ell(\lambda)}\mathbb P\left(\sigma_n\in A_{1, \ldots, \ell( \lambda)}^{\lambda,x_1}\right) +\sum_{\mu \subset \lambda, \lambda\neq\mu} C'_{\lambda,\mu}n^{\ell(\mu)} \mathbb P\left(\sigma_n\in A_{1, \ldots, \ell( \mu)}^{\mu,x_1} \right)\right),
\end{align*}
where in the last equality, we have used \eqref{fillingnb}.
By replacing $\sigma_n$ by $\rho_n$ in the previous calculus, the first implication is direct and the other implication can be obtained by iteration on $\ell(\lambda).$

\end{proof}

We give now the counterpart of Proposition \ref{prop:uniform}.

\begin{proposition} 
Let $\Omega$ a word  which is not a power\color{black}. Suppose that its reduced form  $x_{\alpha_\ell}^{\beta_\ell}\ldots  x_{\alpha_1}^{\beta_1},$ satisfies $\alpha_\ell \neq \alpha_1.$     Assume that $\rho_{1,n}, \ldots, \rho_{k,n}$ are independent and \textbf{uniformly distributed} on $\mathfrak S_n.$ Then,  for any $\hat{g}_1,\hat{g}_2,\dots \hat{g}_k \in  \widehat{\mathbb{G}},$ for any Young diagram $\mu$,
\begin{multline}\label{Prop:mais_graphs_moments}
    \lim_{n\to\infty} n^{\ell(\mu)}\mathbb{P}\left( \forall i \in \llbracket  k \rrbracket, \widehat{\mathcal{G}}^{\llbracket \ell(\mu)\rrbracket, \Omega}_{i}(\rho_{1,n}, \ldots, \rho_{k,n})= \hat{g}_i \text{ and } (\rho_{1,n}, \ldots, \rho_{k,n}) \in A_{1, \ldots, \ell(\mu)}^{\mu,\Omega}  \right)\\ =\prod_{i=1}^k \mathbbm{1}_{\{\hat{g}_i=\mathcal {T}_{|\mu|\cdot\gamma^i(\Omega)}\}},   
\end{multline} 
where we recall that  $|\mu|=\sum_{j=1}^{\ell(\mu)} \mu_j.$ 
Consequently, if
 $\sigma_{1,n}, \ldots, \sigma_{k,n}$ are independent, with conjugation invariant distribution satisfying Assumptions \eqref{hyp:prim} and \eqref{hyp:prim_2}, then, for any  Young diagram $\mu,$  we have that 
\begin{equation}\label{eq:general-comp}
    \mathbb P\left((\sigma_{1,n}, \ldots, \sigma_{k,n}) \in  A_{1, \ldots, \ell(\mu)}^{\mu,\Omega}  \right) = \mathbb P\left((\rho_{1,n}, \ldots, \rho_{k,n} ) \in  A_{1, \ldots, \ell(\mu)}^{\mu,\Omega}  \right) (1+o(1)).
\end{equation}

\end{proposition}

\begin{proof}
A direct consequence of Proposition \ref{prop:pp} and Lemma~\ref{Expectation-probability-conversion} is that for any $\Omega$ which is not a power,
\[
 \mathbb{P}\left(  (\rho_{1,n}, \ldots, \rho_{k,n}) \in A_{1, \ldots, \ell(\mu)}^{\mu,\Omega}  \right) =
 \mathbb{P}\left(  \rho_{1,n}\in A_{1, \ldots, \ell(\mu)}^{\mu,x_1}  \right)  (1+o(1)).  \] 
Since,  when $n>|\mu|$,
\[
 \mathbb{P}\left(  \rho_{1,n}\in A_{1, \ldots, \ell(\mu)}^{\mu,x_1}  \right) =  \frac{(n-\ell(\mu))!}{n!}, \]
 we have, 
 
  \[\lim_{n\to\infty} n^{\ell(\mu)}\mathbb{P}\left(  (\rho_{1,n}, \ldots, \rho_{k,n}) \in A_{1, \ldots, \ell(\mu)}^{\mu,\Omega}  \right)=1. \]

  If $r$ is the total length of the word $\Omega$ written under its canonical form and  $t=r\cdot|\mu|-\ell(\mu)$,
  let $j_1, \ldots, j_t$ be $t$ indices in $\{\ell(\mu)+1, \ldots, n\}$ which are two by two distinct. Let   $(g_1,\dots,g_k)$  be equal to  $(\mathcal G_1^{\llbracket \ell(\mu)\rrbracket, \Omega}(\boldsymbol\sigma),\dots,\mathcal G_k^{\llbracket \ell(\mu)\rrbracket, \Omega}(\boldsymbol\sigma))$ when

 \[ i_0^{1, \Omega^{\mu_1 }}(\boldsymbol\sigma)=1,  i_1^{1, \Omega^{\mu_1 }}= j_1 , , \ldots, i_{\mu_1r}^{1,  \Omega^{\mu_1 }}(\boldsymbol\sigma)=1,\]
 \[ i_0^{2,  \Omega^{\mu_2 }}(\boldsymbol\sigma)=2,  i_1^{2, \Omega^{\mu_2 }} =j_{\mu_1 r} ,  \ldots, i_{\mu_2r}^{2,  \Omega^{\mu_1 }}(\boldsymbol\sigma)=2,
 \]
 \[\dots\]
 \[ i_0^{\ell(\mu),  \Omega^{\mu_{\ell(\mu)} }}(\boldsymbol\sigma)=\ell(\mu),  i_1^{\ell(\mu), \Omega^{\mu_{\ell(\mu)} }}(\boldsymbol\sigma)=j_{1+\sum_{i=1}^{\ell(\mu)-1}(r\cdot\mu_i -1)} , \ldots, i_{\mu_{\ell(\mu)}r}^{1, \Omega^{\mu_{\ell(\mu)} }}(\boldsymbol\sigma)=\ell(\mu).\]

For such a $k$-tuple of graphs, by Lemma \ref{lem:straight2} and independence, we have the following lower bound :
\begin{multline*}
 \mathbb P\left( \forall i  \in \llbracket k\rrbracket, \rho_{i,n} \in  \mathfrak{S}_{n,g_i}\right) \\ \ge  \prod_{i=1}^k \frac{(n-|\mu|\ell_i-|\mu|\sum_{j=1}^{\ell_i} \gamma_j^i(\Omega))!}{(n-|\mu|\ell_i)!}  \left(1- |\mu|\sum_{j=1}^{\ell_i} \mathbb{P}(c_1(\rho_{n,i}) \le  \gamma_j^i(\Omega)) - \frac{|\mu|\ell_i-1}{n-1}  |\mu|\sum_{j=1}^{\ell_i} \gamma_j^i(\Omega) \right),  
\end{multline*}
where $\ell_i$ is the number of non-trivial connected components of $\gamma^i(\Omega).$

From there, as in proof of proposition \ref{prop:uniform}, we get that 
\[ \liminf_{n\to \infty}n^{\ell(\mu)}\mathbb{P}\left( \forall i \in \llbracket  k \rrbracket, \widehat{\mathcal{G}}^{\llbracket \ell(\mu)\rrbracket, \Omega}_{i}(\rho_{1,n}, \ldots, \rho_{k,n})= \mathcal {T}_{|\mu|\cdot\gamma^i(\Omega)}\text{ and } (\rho_{1,n}, \ldots, \rho_{k,n}) \in A_{1, \ldots, \ell(\mu)}^{\mu,\Omega}  \right) \geq 1.
\]%where the last inequality comes from the fact that 
%$ \sum_{j=1}^{\ell_i} \gamma_j^i(\Omega) = \frac r d.$
This concludes the proof of \eqref{Prop:mais_graphs_moments}. Using similar arguments as in the proof of Theorem \ref{th:general}, we get that if $\sigma_n$ satisfies Assumptions \eqref{hyp:prim} and \eqref{hyp:prim_2}, then  \eqref{eq:general-comp} holds.

 \end{proof}

We can now conclude the proof of Theorem \ref{th:conv_dis} itself.

\begin{proof}
We assume that $w= \Omega^d,$ with $d \ge 1$ and $\Omega$ is not a power. Let $d' \in \mathbb N^*$ be fixed. If $\boldsymbol \sigma_n := (\sigma_{1,n}, \ldots, \sigma_{k,n}) $
satisfy \eqref{hyp:prim} and  $\boldsymbol \rho_n := (\rho_{1,n}, \ldots, \rho_{k,n}) $ are independent and uniformly distributed, we want to show that \eqref{eq:general-comp} implies that for any monomial $P,$ 
\begin{equation}\label{eq:monomial-comp}
\lim_{n \rightarrow \infty } \mathbb E\left( P(\#_1 w(\boldsymbol \sigma_n), \ldots, \#_{d'} w(\boldsymbol \sigma_n))\right) = \lim_{n \rightarrow \infty } \mathbb E\left( P(\#_1 w(\boldsymbol \rho_n), \ldots, \#_{d'} w(\boldsymbol \rho_n))\right) .
\end{equation}
Indeed, it was proved in \cite{Nica} that the right handside depends only on $d$ and $P.$
The first remark is that, for any monomial $P,$ there exists a polynomial $Q,$ such that for any fixed $ \boldsymbol\sigma := (\sigma_{1}, \ldots, \sigma_{k}),$
\[ P(\#_1 w(\boldsymbol \sigma), \ldots, \#_{d'} w(\boldsymbol \sigma))
= Q (\#_1 \Omega(\boldsymbol \sigma), \ldots, \#_{dd'} \Omega(\boldsymbol \sigma)).\]
Indeed, for any $j \in\mathbb N^*,$

\begin{align*}
    \#_1 w(\boldsymbol \sigma)^j & = \sum_{i=1}^n \mathbf 1_{w(\boldsymbol \sigma)^j(i)=i} = \sum_{i=1}^n \mathbf 1_{c_i(w(\boldsymbol \sigma)))| j} \\
    & = j   \#_j w(\boldsymbol \sigma) + \sum_{r|j, r\neq j}  r  \#_r w(\boldsymbol \sigma).
\end{align*}
On the other hand, 
\[   \#_1 w(\boldsymbol \sigma)^j =   \#_1 \Omega(\boldsymbol \sigma)^{dj} =  \sum_{r|dj}  r  \#_r \Omega(\boldsymbol \sigma).\]
Therefore, by induction, for any $j \in \mathbb N^*,$ $ \#_j w(\boldsymbol \sigma)$ can be expressed as a linear combination of $\{ \#_r \Omega(\boldsymbol \sigma) \}_{r|dj}$
and it is enough to show Theorem \ref{th:conv_dis}  in the particular case when $d=1$.
As a consequence of Lemma  \ref{Expectation-probability-conversion}, we have that, if for any Young diagram $\mu$, $$ \lim_{n\to\infty} n^{\ell(\mu)}  \mathbb P\left( \Omega(\boldsymbol\sigma_n) \in  A_{1, \ldots, \ell(\mu)}^{\mu,x_1}  \right) = \lim_{n\to\infty} n^{\ell(\mu)} \mathbb{P}\left(\Omega(\boldsymbol\rho_n) \in  A_{1, \ldots, \ell(\mu)}^{\mu, x_1}   \right),$$
then for any $P\in \mathbb{R}[x_1,x_2,\dots,x_{d'}],$ monomial, $$\lim_{n \rightarrow \infty } \mathbb E\left( P(\#_1 \Omega(\boldsymbol\sigma_n), \ldots, \#_{d'} \Omega(\boldsymbol\sigma_n))\right) = \lim_{n \rightarrow \infty } \mathbb E\left( P(\#_1 \Omega(\boldsymbol\rho_n), \ldots, \#_{dd'} \Omega(\boldsymbol\rho_n))\right).$$
The convergence  of joint moments is therefore a direct consequence of \eqref{eq:general-comp} and the convergence in distribution follows.
\end{proof}
We now go to the proof of Corollary \ref{corollary_dis}. Assume that $(\sigma_{1,n}, \ldots, \sigma_{k,n})$ satisfies the assumptions of 
Corollary \ref{corollary_dis}. For every $j \le k,$ let $\tau_{j,n}$ be a random permutation, independent of $(\sigma_{1,n}, \ldots, \sigma_{k,n})$, with $\mathrm{Ewens}_n(0)$ distribution (that is the uniform law on the subset of $\mathfrak S_n$ of permutations having exactly one cycle). We now define, for every $j \le k,$
\[ 
\widetilde \sigma_{j,n} := \left\{
\begin{array}{ll}
\sigma_{j,n}, & \textrm{ if } \forall i \in \widehat{\mathcal N}_j(w), \frac{\#_i \sigma_{j,n}}{\sqrt n} \le 1,\\
\tau_{j,n}, & \textrm{ otherwise.}
\end{array}
\right.
\]
Then $(\widetilde\sigma_{1,n}, \ldots, \widetilde\sigma_{k,n})$ satisfies the assumptions of Theorem \ref{th:conv_dis} and 
$(\#_{1} w(\boldsymbol\sigma_{n}),\dots,\#_{d^\prime}  w(\boldsymbol\sigma_{n}))$
and \linebreak$(\#_{1} w(\widetilde{\boldsymbol\sigma}_{n}),\dots,\#_{d^\prime}  w(\widetilde{\boldsymbol\sigma}_{n}))$
 have asymptotically the same distribution.

\section{Discussion about optimality}\label{sec:opt}
% \begin{rem} \label{rm1}
% One can see that the condition $H_1$ can be replaced by the following condition.
% 
% $\hat{H}_1 :$ For any $k_1,k_2\geq1$, for any $\varepsilon >0$, there exists $n_0$ such  that for any $n>n_0$, for any $g_1\in \mathbb{G}^n_{k_1}$, $g_2\in \mathbb{G}^n_{k_2}$,
% \begin{align*}
% (1-\varepsilon) \p(\sigma_n \in \mathfrak{S}_{n,g_1})\p(\rho_n \in \mathfrak{S}_{n,g_2})&\leq \p(\sigma_n \in \mathfrak{S}_{n,g_1},\rho_n \in \mathfrak{S}_{n,g_2})\leq (1+\varepsilon)\p(\sigma_n \in \mathfrak{S}_{n,g_1})\p(\rho_n \in \mathfrak{S}_{n,g_2}).  
% \end{align*}
% When both permutations are conjugation invariant, we don't need a uniform bound. One can have both bound for every finite graph and  by adding trivial edges to get higher dimensions. \color{red} difficile à comprendre \color{black}
% \end{rem}
% 
%  \color{red} à finir \color{black}
% 
% \begin{rem}\label{opt}

In this last section, we make a few remarks on the optimality of our conditions \eqref{hyp:prim}, \eqref{hyp:prim_2}, \eqref{hyp:main} and \eqref{hyp:main_2} on short cycles. We  hereafter only consider the case when  the permutations are  independent and have conjugation invariant distributions. In several cases, in particular the commutator, We can claim that these conditions are sharp. In \cite{KaMa20} we already discussed the case of the product.

\subsection{Optimality for the commutator}

\begin{itemize}
 \item Assumption  \eqref{hyp:prim} is optimal in the sense that if, for some $\ell \ge 1,$ we have 
\[ \liminf_{n\to\infty} n^{-\frac \ell 2}\min(\mathbb{E}((\#_1\,\sigma_n)^\ell),\mathbb{E}((\#_1\,\rho_n)^\ell))=\varepsilon_\ell>0, \]
\[
\textrm{ then } 
         \liminf_{n\to\infty} \mathbb{E}( (\#_1([\sigma_{n},\rho_n])^\ell) \geq \mathbb{E}(\xi_1^\ell)+\varepsilon^2_\ell.\]
Indeed, one can see that if $g$ is the class of the graph  with adjacency matrix
        $ {\rm  I}_\ell,$ the event \linebreak $\{ (\widehat{\mathcal{G}}^{\llbracket \ell \rrbracket,[x_1,x_2]}_1(\sigma_n,\rho_n), \widehat{\mathcal{G}}^{\llbracket \ell \rrbracket,[x_1,x_2]}_2(\sigma_n,\rho_n))=(\hat{g},\hat{g})\}$ will contribute to the limit, leading to the term $\varepsilon^2_\ell$ in the limit.
\item  Similarly, Assumption  \eqref{hyp:prim_2} is optimal in the sense that if 
\[\liminf_{n\to\infty} \left(\frac{\min(\mathbb{E}(\#_2\,\sigma_n),\mathbb{E}(\#_2\,\rho_n))}{n}\right)=\varepsilon^\prime >0,
\textrm{ then }
         \liminf_{n\to\infty} \mathbb{E}\left( \left(\#_1([\sigma_{n},\rho_n])\right)^2\right) \geq 2+{\varepsilon'}^4.\]
Indeed, as above,  if $\hat{g}^\prime$ is the class of the graph with adjacency matrix 
        $\left(\begin{matrix}
0 & 1 & 0 & 0 \\ 
1 & 0 & 0 & 0 \\
0 & 0 & 0 & 1 \\ 
0 & 0 & 1 & 0 
\end{matrix}\right),$ the event $\{ (\widehat{\mathcal{G}}^{\{1,2\},[x_1,x_2]}_1(\sigma_n,\rho_n), \widehat{\mathcal{G}}^{\{1,2\},[x_1,x_2]}_2(\sigma_n,\rho_n))=(\hat{g}^\prime,\hat{g}^\prime)\}$ will contribute to the limit. 
\end{itemize}

\subsection{Non-optimality in the general framework}

On the other hand, one can find words for which  the conditions are not optimal.
We mention hereafter a few examples where, by easy considerations, sometimes using our previous results on the product, one could improve the assumptions on short cycles.

\begin{itemize}
    \item Take for example the case $w=x_1x_2^3$. 
Using our conditions in the case of the product, one could give conditions on fixed points and two cycles of $\sigma_2^3$, that is conditions on cycles of length $1,3,6$ on $\sigma_2$. Our theorem gives conditions on cycles of lengths, $1,2,3,6$ therefore being suboptimal.
\item  More generally, if $w=w_1w_2$ where $w_1$ and $w_2$ have disjoint supports, one can try to apply the product theorem to get weaker conditions. For example, for $w=(x_1x_2)x_3$ asking that the  number of fixed points of $x_1x_2$ are less than $\varepsilon\sqrt{n}$ and the number of  two-cycles of $x_1x_2$ is less that $\varepsilon n$ can be obtained with only conditions on fixed points of $\sigma_{1,n}$ and $\sigma_{2,n}$.
 \item In a similar spirit,
if for some $j \le k,$ $\gcd{ \mathcal N'_j}=d \neq 1$ for some $j$, one can obtain in many cases better conditions by considering   $x_j^d$ as a new subword  $y$ in $w.$ %one can have first conditions on the cycles of $\sigma_{j,n}^d$ and then translate them in terms of cycles of  $\sigma_{j,n}.$
\end{itemize}
The table below summarizes some cases where we checked whether our conditions are optimal or not. 
{
\begin{table}[!h] \centering
\begin{tabular}{|l|c|}
\hline
$w$                 & Optimality of our conditions \\ \hline
$x_1x_2$            & Yes                          \\ \hline
$x_1x_2^2$          & Yes                          \\ \hline
$[x_1,x_2]$         & Yes                          \\ \hline
$x_1x_2^3$          & No                           \\ \hline
$x_1x_2x_3$         & No                           \\ \hline
$x_1x_2x_1x_2$      & Yes                          \\ \hline
$x_1^2x_2^{-2}$       & Yes                          \\ \hline
$x_3x_1x_3^{-1}x_2$ & No                           \\ \hline
\end{tabular}
\end{table}}

% According to [Nica], Theorem~\ref{th:conv_dis} is true for the uniform case. To prove the general case need to introduce a new notation. For $I \subset [n]$, let $\mathcal G_i^{I, w(\sigma)}$ be the graph of edges $\cup_{m\in I } E_{\mathcal G_i^{m, w(\sigma)}}.$
%\bibliographystyle{}
\end{document}